\documentclass[hidelinks,onefignum,onetabnum]{siamart251216}

\usepackage[T1]{fontenc}
\usepackage[utf8]{inputenc}

\usepackage{amsmath,amssymb,mathtools}

\usepackage{enumitem}
\usepackage{array,booktabs}
\usepackage{xcolor}

\hypersetup{
  colorlinks=true,
  linkcolor=black,
  citecolor=black,
  urlcolor=black
}

\makeatletter
\def\refstepcounter@optarg[#1]#2{%
  \cref@old@refstepcounter{#2}%
  \cref@constructprefix{#2}{\cref@result}%
  \@ifundefined{cref@#1@alias}%
    {\def\@tempa{#1}}%
    {\def\@tempa{\csname cref@#1@alias\endcsname}}%
  \protected@edef\cref@currentlabel{%
    [\@tempa][\arabic{#2}][\cref@result]%
    \csname p@#2\endcsname\csname the#2\endcsname}%
}
\makeatother

\usepackage{xparse}

\newsiamthm{innerthm}{Theorem}
\newsiamthm{innerlem}{Lemma}
\newsiamthm{innerprop}{Proposition}
\newsiamthm{innercor}{Corollary}
\newsiamthm{innerdef}{Definition}
\newsiamthm{innerasm}{Assumption}
\newsiamthm{innerex}{Example}
\newsiamremark{innerrem}{Remark}

\RenewDocumentEnvironment{theorem}{m m}%
  {\begin{innerthm}[{#1}]\label{thm:#2}}%
  {\end{innerthm}}
\RenewDocumentEnvironment{lemma}{m m}%
  {\begin{innerlem}[{#1}]\label{lem:#2}}%
  {\end{innerlem}}
\RenewDocumentEnvironment{proposition}{m m}%
  {\begin{innerprop}[{#1}]\label{prop:#2}}%
  {\end{innerprop}}
\RenewDocumentEnvironment{corollary}{m m}%
  {\begin{innercor}[{#1}]\label{cor:#2}}%
  {\end{innercor}}
\RenewDocumentEnvironment{definition}{m m}%
  {\begin{innerdef}[{#1}]\label{def:#2}}%
  {\end{innerdef}}
\NewDocumentEnvironment{assumption}{m m}%
  {\begin{innerasm}[{#1}]\label{asm:#2}}%
  {\end{innerasm}}
\NewDocumentEnvironment{example}{m m}%
  {\begin{innerex}[{#1}]\label{ex:#2}}%
  {\end{innerex}}
\NewDocumentEnvironment{remark}{m m}%
  {\begin{innerrem}[{#1}]\label{rem:#2}}%
  {\end{innerrem}}

\DeclareMathOperator{\dom}{dom}
\DeclareMathOperator{\gph}{gph}
\DeclareMathOperator{\cone}{cone}
\DeclareMathOperator{\Min}{Min}
\DeclareMathOperator{\Int}{int}
\DeclareMathOperator{\cl}{cl}
\DeclareMathOperator*{\Limsup}{Lim\,sup}
\DeclareMathOperator*{\Liminf}{Lim\,inf}
\DeclareMathOperator{\dist}{dist}

\newcommand{\R}{\mathbb{R}}
\newcommand{\N}{\mathbb{N}}

\newcommand{\DD}{D^{2}}               

\newcommand{\sm}{\rightrightarrows}

\title{Second-Order Sensitivity of Efficient Solution Maps\\
       in Parametric Vector Optimization with Set Constraints}
\author{N.~X.~D. Bao\thanks{Independent Researcher, Ho Chi Minh City, Vietnam
  (\email{nxdbao@gmail.com}).}
\and Tan H. Cao\thanks{Department of Applied Mathematics and Statistics,
  SUNY (State University of New York) Korea, Yeonsu-Gu, Incheon, Korea
  (\email{tan.cao@stonybrook.edu}).}}
\headers{Second-order sensitivity of efficient solutions}{N.~X.~D. Bao and T.~H. Cao}
\date{September 9, 2026}

\begin{document}
\maketitle

\begin{abstract}
We study second-order fixed-ray Dini sensitivity of the efficient solution
map \(S\) for the parametric vector problem \(\min_C f(p,x)\) subject to
\(x\in H(p)\). Under a value-to-decision error bound \textup{(VDB)}\,---\,an
inverse stability estimate\,---\,the second-order semi-derivative of the
marginal map \(\Phi\) transfers to an exact second-order Dini formula for \(S\). The result separates value sensitivity from the inverse recovery of efficient decisions. This mechanism allows a
nonsingleton efficient-value fiber and requires neither injectivity nor
strict monotonicity of \(\nabla_x f\). For structured feasible maps
\(H(p)=\{x\in\Omega:g(p,x)\in D\}\), we derive primitive-data formulas for
\(D^{2}H\), \(D^{2}\Phi\), and \(D^{2}S\) under Robinson metric regularity,
second-order regularity of \(\Omega\) and \(D\), objective-aware
recession conditions, and directional
second-order semi-derivability of the data. Polyhedral
inequality/equality systems give explicit specializations. A parametric
multi-objective portfolio family verifies \textup{(VDB)} with an explicit
constant, obtained directly from the model's feasible directions, while a two-branch example shows how the formula operates on a
nonsingleton decision fiber.
\end{abstract}



\section{Introduction}\label{sec:intro}

The marginal map \(\Phi\) records efficient objective values, whereas the
efficient solution map \(S\) records the decisions that attain them. A
second-order formula for \(\Phi\) therefore does not by itself determine
the second-order movement of \(S\). This distinction matters in robust
multi-objective portfolio models, where parameter shocks must be traced
back to portfolio weights
\cite{GarciaBernabeu2024,Kovalenko2025,MohsenyTonekabony2025,Varmaz2024}.

First-order sensitivity formulas for solution and marginal maps appear in
Luc, Soleimani-Damaneh, and Zamani~\cite{Luc2018SIAM} and, for parametric
generalized equations, in Huy and Lee~\cite{HuyLee2008}. Related
value-function results under relaxed constraint qualifications are given
in \cite{Bondarevsky2016,MinchenkoStakhovski2011,MinchenkoTarakanov2015}.
Bao, Khanh, and Tung~\cite{Bao2025JOTA} derive fixed-ray first- and
second-order Dini formulas for the feasible image and marginal maps. This
work lies within the general sensitivity framework of
\cite{AubinFrankowska1990,BonnansShapiro2000,RockafellarWets1998} and
Robinson-type regularity
\cite{ChieuYaoYen2010,DontchevRockafellar2009,Robinson1976,Robinson1979}.

Decision-level second-order results use proto-, contingent-, and related
derivatives; see Tung~\cite{Tung2018}, Pham~\cite{Pham2022,Pham2023},
Pham--Nguyen~\cite{PhamNguyen2024}, Chuong~\cite{Chuong2013}, and Li,
Sun, and Zhai~\cite{LiSunZhai2012}. The recurring difficulty is the
converse passage from objective-value variations to decision variations
when several decision directions have the same objective image.
Luc, Soleimani-Damaneh, and Zamani~\cite[Cor.~22]{Luc2018SIAM} obtain
first-order decision recovery under strict efficiency and injectivity of
\(\nabla_x f\) on the relevant derivative set. Tung
\cite[Thm.~3.1]{Tung2018} assumes strict monotonicity of
\(\nabla_x f(\bar p,\bar x)(\cdot)\) in the corresponding second-order
argument.

We replace pointwise identification by the set-valued error bound
\begin{equation*}
 \dist\bigl(x,S(p)\bigr)
 \ \lesssim\
 \dist\bigl(f(p,x),\Phi(p)\bigr).
\end{equation*}
This value-to-decision bound \textup{(VDB)} converts a second-order value
residual into a second-order correction of the decision. It permits
nonsingleton fibers and does not require injectivity or strict
monotonicity of \(\nabla_x f\). The argument is primal and follows
realizing sequences; coderivative and limiting-subdifferential
counterparts are outside the scope of this paper.

For the structured feasible map
\[
 H(p)=\{x\in\Omega:g(p,x)\in D\},
\]
Robinson metric regularity, second-order regularity of \(\Omega\) and
\(D\), and directional second-order semi-derivability of \(f\) and \(g\)
give a formula for \(\DD H\). Objective-aware recession conditions then
give the value-level formula for \(\DD\Phi\), which supplies the input
for the formula for \(\DD S\).

\medskip
\noindent\textbf{The main contributions are the following.}
{\small
\begin{enumerate}[label=(C\arabic*),nosep]
\item a \emph{set-valued value-to-decision transfer}
      \(\DD\Phi\xrightarrow{\textup{(VDB)}}\DD S\), yielding an exact
      second-order Dini formula and semi-derivability of \(S\)
      (Theorems~\ref{thm:DDS-semideriv} and~\ref{thm:DDS-struct}); linear
      \textup{(VDB)} preserves the \(o(t^2)\) rate from value residual to
      decision correction (Remark~\ref{rem:VDB-linear-rate}). The
      base efficient-value fiber may be non-singleton, and no injectivity
      or strict monotonicity of \(\nabla_x f\) is imposed;
\item a \emph{structured realization} for
      \(H(p)=\{x\in\Omega:g(p,x)\in D\}\): primitive-data second-order
      formulas for \(H\), \(\Phi\), and \(S\) under the stated regularity
      conditions (Proposition~\ref{prop:DDH-str} and
      Theorems~\ref{thm:DDPhi-str} and~\ref{thm:DDS-struct}). These
      formulas complement the abstract-map second-order theory of Bao,
      Khanh, and Tung~\cite[Thms.~4.2--4.3]{Bao2025JOTA};
\item verification on explicit models: a three-asset portfolio family
      with global \textup{(VDB)} constant \(\sqrt6\)
      (Example~\ref{ex:VDB-portfolio-recovery}) and the noninjective,
      set-valued recovery mechanism in the two-branch fiber of
      Example~\ref{ex:twobranch-VDB}.
\end{enumerate}%
}

\medskip
\noindent\textbf{Theorem A (value-to-decision transfer).}
Under the second-order semi-derivability hypotheses of
Theorem~\ref{thm:DDS-semideriv} and a value-to-decision error bound at
\((\bar p,\bar x)\), the second-order Dini derivative of \(S\) is
\[
\DD S(\bar p,\bar x,u)(p)
=
\left\{w\in\DD H(\bar p,\bar x,u)(p):
\begin{array}{l}
\nabla_x f(\bar p,\bar x)w+d^2f(\bar p,\bar x)(p,u)\\
\quad\in\DD\Phi(\bar p,\bar y,\bar v)(p)
\end{array}\right\}.
\]
Theorem~A is an introductory version of
Theorem~\ref{thm:DDS-semideriv}; its structured specialization is
Theorem~\ref{thm:DDS-struct}, and Example~\ref{ex:DDS-counter} shows
why \textup{(VDB)} is needed.

The paper is organized as follows. Section~\ref{sec:prelim}
fixes notation and regularity, gives sufficient conditions for \textup{(VDB)},
and reviews prior results. Section~\ref{sec:FPhi-struct} proves the structured
formulas for \(H\) and \(\Phi\); Section~\ref{sec:effsol} develops the
abstract theory for \(S\); and Section~\ref{sec:paramsys} gives its structured
specialization and Robinson-type assumptions. Section~\ref{sec:applications}
treats the portfolio family and two-branch example, and
Section~\ref{sec:conclusion} concludes.
\section{Preliminaries}\label{sec:prelim}

\subsection{Basic notation}\label{subsec:notation}
{Throughout the paper we work in finite-dimensional Euclidean
spaces and follow the notation of \cite{Bao2025JOTA}. Let \(\N\) denote the
natural numbers, and let \(\R^{n}\) be equipped with the standard norm
\(\|\cdot\|\) and inner product \(\langle\cdot,\cdot\rangle\). The open unit
ball of \(\R^{n}\) is denoted by \(B^{n}\), and \(B^{n}(x,r)\) denotes the
open ball centered at \(x\) with radius \(r\). For
\(M\subset\R^{n}\), \(\Int M\), \(\cl M\), and \(\cone M\) denote,
respectively, the interior, closure, and conic hull of \(M\), with
\(0\in\cone M\). We write
\(\dist(x,M):=\inf\{\|x-x'\|\,:\,x'\in M\}\).}

Throughout, \(C\subset\R^{n}\) is a nontrivial pointed closed convex cone,
meaning \(\{0\}\subsetneq C\subsetneq\R^{n}\). We call such a cone
\emph{proper}; this convention does not require \(C\) to be solid
\textup{(}to have nonempty interior\textup{)}. All later Henig-type
constructions use a dilating cone \(K\) that need not equal \(C\). The cone
\(C\) induces the partial order
\[
    y_{1}\leqslant_{C} y_{2} \iff y_{2}-y_{1}\in C
    \qquad(y_{1},y_{2}\in\R^{n}),
\]
with positive polar cone
\(C^{\ast}:=\{c^{\ast}\in\R^{n}\,:\,\langle c^{\ast},y\rangle\geqslant 0\text{ for all }y\in C\}\).

For the parametric problem, \(P=\R^s\), \(X=\R^m\), and \(Y=\R^n\).
The objective is \(f:P\times X\to Y\), and \(H:P\sm X\) is the feasible
decision map. The feasible image, marginal, and efficient solution maps are
\begin{align}
    \mathcal F(p) &:= f(p,H(p)) = \{f(p,x)\,:\,x\in H(p)\},\label{eq:Fdef}\\
    \Phi(p) &:= \Min\nolimits_{C} \mathcal F(p),\label{eq:Phidef}\\
    S(p) &:= \{x\in H(p)\,:\,f(p,x)\in\Phi(p)\}.\label{eq:Sdef}
\end{align}
The notation used throughout the paper is collected below.

\begin{table}[ht]
\centering
\footnotesize
\caption{Main notation and conventions}
\label{tab:notation}
\begin{tabular}{@{}p{0.38\linewidth}p{0.54\linewidth}@{}}
\toprule
Symbol & Meaning \\
\midrule
\(P,X,Y\) & parameter, decision, and objective spaces \\
\(Z\) & constraint codomain, with \(Z:=\R^{q}\) \\
\(C,C^*,K\) & ordering cone, positive polar, and a dilating cone for Henig efficiency \\
\(H(p)\) & feasible decision map \\
\(\mathcal F(p)=f(p,H(p))\) & feasible image map \\
\(\Phi(p)=\Min_C\mathcal F(p)\) & marginal or efficient-value map \\
\(S(p)=\{x\in H(p):f(p,x)\in\Phi(p)\}\) & efficient solution map \\
\(\Min_C A\) & \(C\)-minimal points of \(A\), with \(\Min_C\emptyset=\emptyset\) \\
\(\dist(z,A)\) & distance to \(A\), with \(\dist(z,\emptyset)=+\infty\) \\
\bottomrule
\end{tabular}
\end{table}

For a set-valued mapping \(F:\R^{m}\sm\R^{n}\), the \emph{domain} and
\emph{graph} of \(F\) are
\(\dom F=\{x\in\R^{m}\,:\,F(x)\ne\emptyset\}\) and
\(\gph F=\{(x,y)\in\R^{m}\times\R^{n}\,:\,y\in F(x)\}\), respectively.
The Painlev\'{e}--Kuratowski outer and inner limits of \(F\) are
\begin{align*}
    \Limsup_{x\to\bar x} F(x)
    &= \{y\in\R^{n}\,:\,\exists x_{k}\to\bar x,\exists y_{k}\in F(x_{k}),y_{k}\to y\},\\
    \Liminf_{x\to\bar x} F(x)
    &= \{y\in\R^{n}\,:\,\forall x_{k}\to\bar x,\exists y_{k}\in F(x_{k}),y_{k}\to y\},
\end{align*}
in the sense of \cite[Ch.~1]{AubinFrankowska1990}
and~\cite[Ch.~4]{RockafellarWets1998}. We say \(F\) is \emph{outer (inner)
semicontinuous} at \(\bar x\) if
\(\Limsup_{x\to\bar x}F(x)\subset F(\bar x)\) (respectively,
\(\Liminf_{x\to\bar x}F(x)\supset F(\bar x)\)). \(F\) is outer
semicontinuous at every \(x\in\R^{m}\) if and only if \(\gph F\) is closed,
cf.~\cite[Prop.~1.4.4]{AubinFrankowska1990} and~\cite[Thm.~5.7]{RockafellarWets1998}.

For a set \(M\subset\R^{n}\) with \(\bar x\in M\), the \emph{contingent
cone} is \(T(M,\bar x)=\Limsup_{t\downarrow 0}(M-\bar x)/t\).
Second-order tangent-like sets, introduced by Ben-Tal and
Zowe~\cite{BenTalZowe1982} and developed systematically in
Cominetti~\cite{Cominetti1990}, Penot~\cite{Penot1998} and
Bonnans--Shapiro~\cite[Sect.~3.2]{BonnansShapiro2000}, are recalled below.

\begin{definition}{Second-order tangent sets}{2ndtan}
Let \(\bar x\in M\subset\R^{n}\) and \(u\in\R^{n}\).
\begin{enumerate}[label=(\roman*)]
\item The \emph{second-order contingent set} of \(M\) at \(\bar x\) in
      direction \(u\) is
      \[
          T^{2}(M,\bar x,u)=\Limsup_{t\downarrow 0}\frac{M-\bar x-tu}{t^{2}}.
      \]
\item The \emph{second-order adjacent set} is
      \[
          T^{\flat 2}(M,\bar x,u)=\Liminf_{t\downarrow 0}\frac{M-\bar x-tu}{t^{2}}.
      \]
\item The \emph{second-order asymptotic contingent set} is
      \[
          T''(M,\bar x,u)=\Limsup_{(t,r)\downarrow(0,0),\,t/r\to 0}\frac{M-\bar x-tu}{tr}.
      \]
\end{enumerate}
\(M\) is \emph{second-order regular} at \(\bar x\) in direction
\(u\) if \(T^{2}(M,\bar x,u)=T^{\flat 2}(M,\bar x,u)\).
Convex polyhedral sets are second-order regular, and in that case
\(T^{2}(M,\bar x,u)=T(T(M,\bar x),u)\)
\cite[Prop.~3.34]{BonnansShapiro2000}.
\end{definition}

Dini directional derivatives for set-valued mappings are defined
next. The first-order form is classical
\cite[Ch.~5]{AubinFrankowska1990},~\cite[Ch.~8]{RockafellarWets1998}, used
in \cite[Sect.~2]{Luc2018SIAM}; the second-order form follows
\cite{Bao2025JOTA,KhanhTung2015,LiSunZhai2012}.

\begin{definition}{Dini directional derivatives}{Dini}
Let \(F:\R^{m}\sm\R^{n}\), \((\bar x,\bar y)\in\gph F\) and
\((u,v)\in\R^{m}\times\R^{n}\).
\begin{enumerate}[label=(\roman*)]
\item The \emph{first-order upper and lower Dini derivatives} of \(F\) at
      \((\bar x,\bar y)\) in direction \(u\) are
            \begin{align*}
          DF(\bar x,\bar y)(u)
          &=\Limsup_{t\downarrow 0}\frac{F(\bar x+tu)-\bar y}{t},\\
          D_{\ell}F(\bar x,\bar y)(u)
          &=\Liminf_{t\downarrow 0}\frac{F(\bar x+tu)-\bar y}{t}.
      \end{align*}
      
      \(F\) is \emph{semi-derivable} at \((\bar x,\bar y)\) in direction
      \(u\) if \(DF(\bar x,\bar y)(u)=D_{\ell}F(\bar x,\bar y)(u)\).
\item The \emph{second-order upper and lower Dini derivatives} of \(F\) at
      \((\bar x,\bar y)\) in direction \((u,v)\) are
      \begin{align*}
          \DD F(\bar x,\bar y,v)(u)
          &=\Limsup_{t\downarrow 0}\frac{F(\bar x+tu)-\bar y-tv}{t^{2}},\\
          \DD_{\ell} F(\bar x,\bar y,v)(u)
          &=\Liminf_{t\downarrow 0}\frac{F(\bar x+tu)-\bar y-tv}{t^{2}}.
      \end{align*}
      \(F\) is \emph{second-order semi-derivable} at \((\bar x,\bar y)\) in
      direction \((u,v)\) if \(\DD F=\DD_{\ell}F\) at that point and
      direction.
\end{enumerate}
\end{definition}

For a differentiable single-valued map \(f:\R^{m}\to\R^{n}\), the
second-order upper Dini derivative collapses to the second-order radial
derivative
\[
    \DD f(\bar x)(u)=\Limsup_{t\downarrow 0}\frac{f(\bar x+tu)-f(\bar x)-t\nabla f(\bar x)u}{t^{2}}.
\]
When \(f\) is second-order semi-derivable at \(\bar x\) in direction
\(u\), this set reduces to a single element of \(\R^{n}\), denoted by
\[
    d^{2}f(\bar x)(u)
    :=\lim_{t\downarrow 0}\frac{f(\bar x+tu)-f(\bar x)-t\nabla f(\bar x)u}{t^{2}}.
\]
In particular, when \(f\) is twice differentiable at \(\bar x\),
\[
 d^{2}f(\bar x)(u)
 =\tfrac12\,\nabla^2_{\!\mathrm{Hess}}f(\bar x)[u,u],
\]
where \(\nabla^2_{\!\mathrm{Hess}}f\) denotes the standard Hessian
bilinear form. See \cite[Rem.~2.2]{Bao2025JOTA}.

\begin{remark}{Second-order scaling convention}{2nd-order-conv}
Second-order expansions in this paper are normalized by \(t^2\), not by
\(\tfrac12t^2\); see Definitions~\ref{def:2ndtan} and~\ref{def:Dini}.
The symbol \(d^2f(\bar x)(u)\) always denotes the resulting Taylor
coefficient:
\[
 f(\bar x+tu)
 =f(\bar x)+t\nabla f(\bar x)u+t^2d^2f(\bar x)(u)+o(t^2).
\]
If \(f\) is \(C^2\), then
\[
 d^2f(\bar x)(u)
 =\tfrac12\,\nabla^2_{\!\mathrm{Hess}}f(\bar x)[u,u].
\]
Thus \(\nabla^2_{\!\mathrm{Hess}}f\) is reserved for the standard Hessian,
whereas \(d^2f\) denotes the coefficient compatible with the Dini
derivatives used here. The same convention applies to \(d^2g\).
\end{remark}

\subsection{Efficiency and stability notions}\label{subsec:regularity}
The recalled results below use standard pseudo-Lipschitz, directional
calmness, and directional compactness notions. We follow
\cite[Sect.~2]{Luc2018SIAM} and
\cite[Defs.~2.1--2.3]{Bao2025JOTA}; see also
\cite[Chs.~3--4]{DontchevRockafellar2009}.

The three solution notions below are used throughout the paper.
Efficient points are the classical Pareto-type notion of
\cite[Ch.~2]{Luc1989}. Henig efficiency
\cite{Henig1982} refines them via enlarged cones, with the uniform
variant of \cite[Def.~3.1]{Bao2025JOTA} and
\cite[Def.~1]{Luc2018SIAM} encoding robustness
under local perturbations.

\begin{definition}{Efficient, Henig and uniformly Henig efficient points}{henig}
Let \((\bar x,\bar y)\in\gph F\).
\begin{enumerate}[label=(\roman*)]
\item \(\bar y\) is an \emph{efficient point} of \(F(\bar x)\) if
      \((F(\bar x)-\bar y)\cap(-C\setminus\{0\})=\emptyset\).
\item For a convex cone \(K\subset\R^{n}\) dilating
      \(C\) (meaning \(C\setminus\{0\}\subset\Int K\)),
      \(\bar y\) is a \emph{Henig \(K\)-efficient point} of
      \(F(\bar x)\) if \((F(\bar x)-\bar y)\cap(-K)=\{0\}\).
\item \(\bar y\) is a \emph{uniformly Henig efficient point} of \(F\) at
      \(\bar x\) if
{\(\bar y\in\Min_{C}F(\bar x)\) and}
      there exist \(K\) dilating \(C\) and \(\delta>0\) such
      that, for all \((x,y)\in B^{m+n}((\bar x,\bar y),\delta)\) with
      \(y\in\Min_{C}F(x)\), \(y\) is a Henig \(K\)-efficient point of
      \(F(x)\). \textup{(}Equivalently, the condition holds on
      \(\gph(\Min_{C}F)\cap B^{m+n}((\bar x,\bar y),\delta)\).\textup{)}
\end{enumerate}
\end{definition}

The uniformity in (iii) is what enables second-order use: a
single dilating cone \(K\) Henig-separates efficient values for all
\((x,y)\) near the base point, so the separation does not degrade at the
\(o(t^{2})\) scale of the second-order expansions.

The next definition supplies the value-to-decision lift used in
Section~\ref{sec:effsol}; a data-level certificate is then provided by
Proposition~\ref{prop:FMR-implies-VDB}.

\begin{definition}{Value-to-decision error bound (VDB)}{vdb}
Let \(\bar x\in S(\bar p)\) and \(\bar y=f(\bar p,\bar x)\in\Phi(\bar p)\).
We say that \((\bar p,\bar x)\) satisfies the \emph{value-to-decision
error bound} if there exist neighborhoods \(V\) of \(\bar p\), \(W\) of
\(\bar x\), and a constant \(\kappa>0\) such that
\begin{equation}\label{eq:VDB}
    \dist\bigl(x,S(p)\cap W\bigr)
    \leqslant \kappa\,\dist\bigl(f(p,x),\Phi(p)\bigr)
    \qquad\forall p\in V,\;x\in H(p)\cap W.
\end{equation}
We also refer to \eqref{eq:VDB} as \textup{(VDB)}.
\end{definition}

\textup{(VDB)} is a localized metric-subregularity-type inverse estimate.
To see this, define the feasible residual map by
\(\Psi_H(p,x):=f(p,x)-\Phi(p)\) for \(x\in H(p)\), and by \(\emptyset\) otherwise.
Then \(S(p)=\{x:0\in\Psi_H(p,x)\}\) and, for feasible \(x\),
\(\dist(0,\Psi_H(p,x))=\dist(f(p,x),\Phi(p))\). Thus \eqref{eq:VDB}
is a metric-subregularity estimate for \(x\mapsto\Psi_H(p,x)\), localized near \(\bar x\) and uniform for \(p\in V\).

\begin{proposition}{Fiber metric regularity implies \textup{(VDB)}}{FMR-implies-VDB}
Let \((\bar p,\bar x)\in\gph S\) and \(\bar y=f(\bar p,\bar x)\in\Phi(\bar p)\), and define
\[
 \Delta(p,y):=\{x\in H(p):f(p,x)=y\}.
\]
Suppose there exist neighborhoods \(V\) of \(\bar p\), \(W\) of \(\bar x\), and \(U\) of \(\bar y\), and constants \(L,r>0\), such that:
\begin{enumerate}[label=\textup{(\roman*)}]
\item \(B^n(\bar y,4r)\subset U\), and \(f\bigl(p,H(p)\cap W\bigr)\subset B^n(\bar y,r)\) for every \(p\in V\);
\item \(\Phi\) is inner semicontinuous at \((\bar p,\bar y)\), with \(V\) chosen so that \(\Phi(p)\cap B^n(\bar y,r)\ne\emptyset\) for every \(p\in V\);
\item \(\dist\bigl(x,\Delta(p,y)\cap W\bigr)\leqslant L\,\|f(p,x)-y\|\) for every \(p\in V\), \(x\in H(p)\cap W\), and \(y\in\Phi(p)\cap U\).
\end{enumerate}
Then \((\bar p,\bar x)\) satisfies \textup{(VDB)} with \(\kappa=L\).
\end{proposition}

\begin{proof}
By hypothesis (ii), after shrinking only \(V\), \(\Phi(p)\cap B^n(\bar y,r)\ne\emptyset\) for every \(p\in V\); this restriction preserves the fiber Lipschitz clause and its constant \(L\).
Fix \(p\in V\), \(x\in H(p)\cap W\), and set
\(\rho:=\dist(f(p,x),\Phi(p))\). The persistence just established
and the value localization in \textup{(i)} imply
\(\rho<2r\).

\emph{Approximate selection.}
For \(0<\varepsilon<r\), choose \(y_\varepsilon\in\Phi(p)\) such that
\[
 \|f(p,x)-y_\varepsilon\|\leqslant\rho+\varepsilon.
\]
Then \(\|y_\varepsilon-\bar y\|<4r\), so
\(y_\varepsilon\in\Phi(p)\cap U\).

\emph{Efficiency propagation.}
If \(z\in\Delta(p,y_\varepsilon)\), then \(z\in H(p)\) and
\(f(p,z)=y_\varepsilon\in\Phi(p)\); hence \(z\in S(p)\) by
\eqref{eq:Sdef}. Thus \(\Delta(p,y_\varepsilon)\subset S(p)\).

\emph{Fiber regularity.}
Hypothesis \textup{(iii)} gives
\[
 \dist\bigl(x,\Delta(p,y_\varepsilon)\cap W\bigr)
 \leqslant L\|f(p,x)-y_\varepsilon\|
 \leqslant L(\rho+\varepsilon).
\]

\emph{Conclusion.}
The preceding inclusion and estimate yield
\[
 \dist\bigl(x,S(p)\cap W\bigr)
 \leqslant\dist\bigl(x,\Delta(p,y_\varepsilon)\cap W\bigr)
 \leqslant L(\rho+\varepsilon).
\]
Letting \(\varepsilon \downarrow 0\) proves \eqref{eq:VDB} with
\(\kappa=L\).
\end{proof}

\begin{remark}{Affine-polyhedral case}{FMR-Hoffman}
Suppose \(\Omega\subset X\) and \(D\subset Z\) are convex polyhedra and both \(f\) and \(g\) are affine in \(x\). Then, for each \((p,y)\), the fiber \(\Delta(p,y)=\{x\in\Omega:g(p,x)\in D,\;f(p,x)=y\}\) is a (possibly empty) convex polyhedron, and hypothesis \textup{(iii)} of Proposition~\ref{prop:FMR-implies-VDB} follows from a Hoffman-type linear error bound for the defining affine inequalities and equations, provided the coefficient matrices admit a Hoffman modulus common to all \(p\in V\)---automatic when they are independent of \(p\)---and \(W\) is shrunk so that the bound localizes. Hypothesis \textup{(ii)}---inner semicontinuity of \(\Phi\) at \((\bar p,\bar y)\)---remains a separate requirement.
\end{remark}
\begin{remark}{Sharp scalarization condition}{VDB-sharp-fwd}
Suppose there are neighborhoods \(V\) of \(\bar p\) and \(W\) of
\(\bar x\), a vector \(\lambda\in C^*\setminus\{0\}\), a function
\(m:V\to\R\), and \(\alpha>0\) such that \(\Phi(p)\ne\emptyset\)
and, for every \(p\in V\),
\begin{align}
 \langle\lambda,y\rangle&=m(p)\qquad(y\in\Phi(p)),
 \label{eq:VDB-common-level}\\
 \alpha\,\dist\bigl(x,S(p)\cap W\bigr)
 &\leqslant \langle\lambda,f(p,x)\rangle-m(p)
 \qquad(x\in H(p)\cap W).
 \label{eq:VDB-sharp}
\end{align}
Choosing an efficient value that nearly attains the value residual and
applying Cauchy--Schwarz gives \textup{(VDB)} with
\(\kappa=\|\lambda\|/\alpha\).
\end{remark}%

The level-boundedness condition below is \((A_1)\); see
\cite[Eq.~\textup{(}20\textup{)}]{Luc2018SIAM} and
\cite[Sect.~4]{Bao2025JOTA}.

\begin{definition}{Locally bounded level sets}{lbls}
We say that \(f\) has \emph{locally bounded level sets} at
\((\bar p,\bar y)\in P\times Y\) relative to \(H\) if there exist
\(\varepsilon>0\) and \(\alpha>0\) such that
\begin{equation}\label{eq:lbls}
    \{x\in H(p)\,:\,f(p,x)\in y-C\}\subset\alpha\cl B^{m}
    \qquad\forall(p,y)\in B^{s+n}((\bar p,\bar y),\varepsilon).
\end{equation}
It is enough for \(H\) to be locally bounded around \(\bar p\).
\end{definition}

\subsection{Prior sensitivity results and relation to the present work}\label{subsec:prior}

The comparison concerns first-order semi-derivability in the
sense of Luc, Soleimani-Damaneh, and Zamani~\cite[Thm.~18,
Cor.~22]{Luc2018SIAM}, and second-order proto-, contingent-, and Dini
analysis in Tung~\cite[Thm.~3.1]{Tung2018}, Pham~\cite{Pham2022,Pham2023},
Pham--Nguyen~\cite{PhamNguyen2024}, and Bao, Khanh, and
Tung~\cite[Thms.~4.1--4.3]{Bao2025JOTA}. The gap addressed here is the
converse passage from value variations to decision variations when
\(\nabla_xf\) is noninjective; Table~\ref{tab:prior} gives the item-by-item
comparison.

For \((\bar p,\bar y)\in\gph\mathcal F\), \(\bar x\in H(\bar p)\), and
\((p,v)\in P\times Y\), we use the branch sets
\begin{align}
    \Delta_{0}(\bar p,\bar y)
        &:= \{x\in H(\bar p)\,:\,f(\bar p,x)=\bar y\},\label{eq:Delta0}\\
    \Delta_{1}(\bar p,\bar x,p,v)
        &:= \{u\in DH(\bar p,\bar x)(p)\,:\,\nabla f(\bar p,\bar x)(p,u)=v\}.\label{eq:Delta1}
\end{align}

\begin{table}[ht]
\centering
\footnotesize
\caption{Prior sensitivity results recalled for comparison, with the
value-to-decision recovery mechanism used in each.}
\label{tab:prior}
\begin{tabular}{@{}>{\raggedright\arraybackslash\vspace{0pt}}p{0.16\linewidth}>{\raggedright\arraybackslash\vspace{0pt}}p{0.18\linewidth}>{\raggedright\arraybackslash\vspace{0pt}}p{0.18\linewidth}>{\raggedright\arraybackslash\vspace{0pt}}p{0.36\linewidth}@{}}
\toprule
Source & Object & Order / derivative & Value-to-decision recovery \\
\midrule
Luc et al.\ \cite{Luc2018SIAM}
& \(S\) & first-order semi-derivative
& Singleton fiber via strict efficiency; injective \(\nabla_x f\). \\
Tung~\cite{Tung2018}
& \(S\) & second-order proto-derivative
& Strict monotonicity of \(\nabla_x f(\bar p,\bar x)\). \\
Pham~\cite{Pham2022,Pham2023} and
Pham--Nguyen~\cite{PhamNguyen2024}
& \(\Phi\), \(S\)
& generalized proto- and contingent derivatives
& Proper-efficiency variants; setting-adapted hypotheses. \\
Bao et al.\ \cite{Bao2025JOTA}
& \(\mathcal F,\Phi\) & first- and second-order Dini derivatives
& Value level only; no decision recovery. \\
Present paper & \(S\) & second-order fixed-ray Dini derivative
& Set-valued via \textup{(VDB)}; no pointwise identification. \\
\bottomrule
\end{tabular}
\end{table}

\section[Second-order theory of H and Phi in the structured setting]{\texorpdfstring{{Second-order theory of \(H\) and \(\Phi\) in the structured setting}}{Second-order theory of H and Phi in the structured setting}}%
\label{sec:FPhi-struct}

\subsection{Standing setup and hypothesis packages}\label{subsec:sec4-setup}

This section derives structured second-order formulas for the feasible map
\begin{equation}\label{eq:Hstruct-def}
    H(p) = \{x\in \Omega\,:\,g(p,x)\in D\},
\end{equation}
where \(\Omega\subset X\) and \(D\subset Z\) are
closed. Throughout this section, \(f\) and \(g\) are continuous, and \(Z:=\R^{q}\) denotes the constraint codomain.

The structured formula complements the abstract results in
\cite[Thms.~4.2--4.3]{Bao2025JOTA} and is proved under the packages in
Table~\ref{tab:sec3-hypotheses}. The resulting formula for \(\DD\Phi\)
supplies the value-level input for \(\DD S\).

We use the Robinson-type error bound below
\cite{Robinson1976,Robinson1979}; see also
\cite{ChieuYaoYen2010,DontchevRockafellar2009}.

\begin{definition}{Robinson metric regularity}{RMR}
Let \((\bar p,\bar x)\in\gph H\). \(H\) is \emph{Robinson
metrically regular along \(\Omega\) at \((\bar p,\bar x)\)} if there exist
neighborhoods \(U_{1}\) of \(\bar p\), \(U_{2}\) of \(\bar x\), and
constants \(\alpha>0\), \(\gamma>0\) such that
\[
    \dist(x,H(p))\leqslant\alpha\,\dist(g(p,x),D)
    \qquad\forall p\in U_{1},\;x\in \Omega\cap U_{2},\;\dist(g(p,x),D)<\gamma.
\]
\end{definition}

The definitions and conditions below form the hypothesis
packages used in the value-level results. Fix a base pair
\((\bar p,\bar y)\) and a direction \((p,v)\).

\begin{proposition}{A scalarized family of dilating cones}{canonical-K}
Let \(\lambda\in\Int C^*\), and set
\[
 m_C(\lambda):=\min\{\langle\lambda,z\rangle:z\in C,\ \|z\|=1\}>0.
\]
For \(0<\varepsilon<m_C(\lambda)\), define
\[
 K_{\lambda,\varepsilon}
 :=\{z\in Y:\langle\lambda,z\rangle\geqslant
                  \varepsilon\|z\|\}.
\]
Then \(K_{\lambda,\varepsilon}\) is a closed pointed convex cone and
\(C\setminus\{0\}\subset\Int K_{\lambda,\varepsilon}\). Moreover, let
\(G\subset\gph(\Min_C F)\) be such that
\[
 \langle\lambda,w-y\rangle\geqslant0
 \qquad\forall (q,y)\in G,\quad\forall w\in F(q).
\]
Then every \((q,y)\in G\) is Henig
\(K_{\lambda,\varepsilon}\)-efficient. In particular, if this common
scalarization condition holds on
\(\gph(\Min_C F)\cap B((\bar q,\bar y),\delta)\), the same cone witnesses
uniform Henig efficiency at \((\bar q,\bar y)\).
\end{proposition}

\begin{proof}
Since \(\lambda\in\Int C^*\), compactness of
\(C\cap\{z:\|z\|=1\}\) gives \(m_C(\lambda)>0\). The defining
inequality makes \(K_{\lambda,\varepsilon}\) a closed convex cone; if
\(z,-z\in K_{\lambda,\varepsilon}\), then
\(0\geqslant2\varepsilon\|z\|\), so it is pointed. Moreover,
\(\langle\lambda,z\rangle\geqslant m_C(\lambda)\|z\|>
\varepsilon\|z\|\) for \(z\in C\setminus\{0\}\), proving the interior
inclusion. Finally, if \(w-y\in-K_{\lambda,\varepsilon}\), the common
scalarization inequality forces \(w=y\); hence
\((F(q)-y)\cap(-K_{\lambda,\varepsilon})=\{0\}\), and the uniform
claim follows from Definition~\ref{def:henig}\textup{(iii)}.
\end{proof}

The hypothesis packages used below are listed atomically in
Table~\ref{tab:sec3-hypotheses}.

\begin{table}[!ht]
\centering
\footnotesize
\caption{Standing hypothesis packages used in
Section~\ref{sec:FPhi-struct}.}
\label{tab:sec3-hypotheses}
\begin{tabular}{@{}l@{\qquad}>{\raggedright\arraybackslash}p{0.82\linewidth}@{}}
\toprule
Tag & Statement \\
\midrule
\((A_1)\) &
The objective \(f\) has locally bounded level sets at
\((\bar p,\bar y)\) relative to \(H\), in the sense of
Definition~\ref{def:lbls}. \\
\((A_2)\) &
For every \(x\in\Delta_0(\bar p,\bar y)\), \(f\) and \(g\) are
continuously differentiable near \((\bar p,x)\), and \(H\) is
semi-derivable at \((\bar p,x)\) in direction \(p\). \\
\((A_3)\) &
For every \(x\in\Delta_0(\bar p,\bar y)\) and every
\(u\in\Delta_1(\bar p,x,p,v)\), \(f\) and \(g\) are second-order
semi-derivable at \((\bar p,x)\) in direction \((p,u)\), \(\Omega\) is
second-order regular at \(x\) in direction \(u\), and \(D\) is
second-order regular at \(g(\bar p,x)\) in direction
\(\nabla g(\bar p,x)(p,u)\). \\
\((A_4)\) &
For every \(x\in\Delta_0(\bar p,\bar y)\), \(H\) is Robinson metrically
regular along \(\Omega\) at \((\bar p,x)\). \\
\((A_5)\) &
For every \(x\in\Delta_0(\bar p,\bar y)\),
\(\{\tilde u\in T(\Omega,x):
\nabla_xg(\bar p,x)\tilde u\in T(D,g(\bar p,x)),\
\nabla_xf(\bar p,x)\tilde u\in-C\}=\{0\}\). \\
\((A_6)\) &
For every \(x\in\Delta_0(\bar p,\bar y)\) and every
\(u\in\Delta_1(\bar p,x,p,v)\),
\(\{\tilde w\in T''(\Omega,x,u):
\nabla_xg(\bar p,x)\tilde w\in
T''(D,g(\bar p,x),\nabla g(\bar p,x)(p,u)),\
\nabla_xf(\bar p,x)\tilde w\in-C\}=\{0\}\). \\
\bottomrule
\end{tabular}
\end{table}
For \(\bar x\in\Delta_0(\bar p,\bar y)\), define the
\emph{first-order objective-aware recession set}
\begin{equation}\label{eq:R1-recession}
 R_1(\bar p,\bar x)
 :=\{\tilde u\in T(\Omega,\bar x):
      \nabla_xg(\bar p,\bar x)\tilde u\in T(D,g(\bar p,\bar x)),\ 
      \nabla_xf(\bar p,\bar x)\tilde u\in-C\}.
\end{equation}
For \(u\in\Delta_1(\bar p,\bar x,p,v)\), define the
\emph{second-order objective-aware recession set}
\begin{equation}\label{eq:R2-recession}
 \begin{aligned}
 R_2(\bar p,\bar x,u,p):={}&\{\tilde w\in T''(\Omega,\bar x,u):
      \nabla_xg(\bar p,\bar x)\tilde w\\
 &\quad\in T''(D,g(\bar p,\bar x),\nabla g(\bar p,\bar x)(p,u)),\ 
      \nabla_xf(\bar p,\bar x)\tilde w\in-C\}.
 \end{aligned}
\end{equation}
Thus \((A_5)\) and \((A_6)\) state, respectively, that
\(R_1(\bar p,\bar x)=\{0\}\) and
\(R_2(\bar p,\bar x,u,p)=\{0\}\) on the branches quantified in the table.

The three second-order objects of this section are chained together as follows:

\medskip
\begin{center}
\footnotesize
$\displaystyle
\underset{\substack{\text{Prop.~\ref{prop:DDH-str}:\ inverse function}\\ (A_2),\,(A_3),\,(A_4)}}{\boxed{\DD H}}
\;\longrightarrow\;
\underset{\substack{\text{Prop.~\ref{prop:DDF-str}:\ chain rule}\\ (A_1),\,(A_2),\,(A_3),\,(A_5),\,(A_6)}}{\boxed{\DD\mathcal F}}
\;\longrightarrow\;
\underset{\substack{\text{Thm.~\ref{thm:DDPhi-str}:\ }\Min_C\\ \text{uniform Henig }K,\ C\text{-minimality}}}{\boxed{\DD\Phi}}
$
\end{center}
\medskip

\noindent Robinson metric regularity \((A_4)\) drives the inverse-function step for \(\DD H\), with \((A_2)\) and \((A_3)\) supplying the required expansions and tangent-set structure. The chain rule then produces \(\DD\mathcal F\), with \((A_1)\) bounding realizing decisions and \((A_5)\), \((A_6)\) excluding first- and second-order objective-aware recession (Lemmas~\ref{lem:blowup-1} and~\ref{lem:blowup-2}). Finally, \(\Min_C\) via uniform Henig efficiency transfers \(C\)\nobreakdash-minimality from \(\mathcal F\) to \(\Phi\).

\subsection[Second-order semi-derivative of H]{\texorpdfstring{{Second-order semi-derivative of \(H\)}}{Second-order semi-derivative of H}}\label{subsec:DDH-struct}

\begin{proposition}{Second-order semi-derivative of \(H\) in the structured setting}{DDH-str}
Let \((\bar p,\bar x)\in\gph H\) and \((p,u)\in P\times X\) satisfy
the \emph{first-order compatibility conditions}
\begin{equation}\label{eq:DDH-str-fo-compat}
    u\in T(\Omega,\bar x)
    \qquad\text{and}\qquad
    \nabla g(\bar p,\bar x)(p,u)\in T(D,g(\bar p,\bar x)).
\end{equation}
Assume:
\begin{enumerate}[label=\textup{(\roman*)},nosep]
\item \(\Omega\) is second-order regular at \(\bar x\) in direction \(u\);
\item \(D\) is second-order regular at \(g(\bar p,\bar x)\) in direction
      \(\nabla g(\bar p,\bar x)(p,u)\);
\item \(g\) is \(C^{1}\) and second-order semi-derivable at
      \((\bar p,\bar x)\) in direction \((p,u)\);
\item \(H\) is Robinson metrically regular along \(\Omega\) at
      \((\bar p,\bar x)\).
\end{enumerate} Then \(H\) is second-order semi-derivable at
\((\bar p,\bar x)\) in direction \((p,u)\), and
\begin{equation}\label{eq:DDH-str}
    \DD H(\bar p,\bar x,u)(p)
    = \left\{w\in T^{2}(\Omega,\bar x,u)\,:\,
       \begin{array}{l}
          \nabla_{x}g(\bar p,\bar x)w + d^{2}g(\bar p,\bar x)(p,u)\\
          \quad\in T^{2}(D,g(\bar p,\bar x),\nabla g(\bar p,\bar x)(p,u))
       \end{array}\right\}.
\end{equation}
\end{proposition}

\begin{proof}
Let \(\mathcal R\) denote the right-hand side of
\eqref{eq:DDH-str}. For \(w\in X\), set
\[
    z := \nabla_{x}g(\bar p,\bar x)w + d^{2}g(\bar p,\bar x)(p,u).
\]
We first record the expansion of \(g\) used in both
inclusions. Let \(t_{k}\downarrow 0\) and let \(w_{k}\to w\).
The \(C^{1}\) mean-value theorem in the second argument gives
\begin{equation}\label{eq:g-mvt}
    g(\bar p+t_{k}p,\bar x+t_{k}u+t_{k}^{2}w_{k})
    = g(\bar p+t_{k}p,\bar x+t_{k}u)
      + t_{k}^{2}\nabla_{x}g(\bar p+t_{k}p,\bar x+t_{k}u)\,w_{k}+r_{k},
\end{equation}
where
\(\|r_{k}\|\leqslant t_{k}^{2}\|w_{k}\|\sup_{\xi}\|\nabla_{x}g(\bar p+t_{k}p,\xi)-\nabla_{x}g(\bar p+t_{k}p,\bar x+t_{k}u)\|\)
with the supremum taken over the segment. Continuity of
\(\nabla_{x}g\) at \((\bar p,\bar x)\) and \(w_{k}\to w\) give
\(r_{k}=o(t_{k}^{2})\). By second-order semi-derivability of
\(g\) at \((\bar p,\bar x)\) in direction \((p,u)\),
\begin{equation}\label{eq:g-sd}
    g(\bar p+t_{k}p,\bar x+t_{k}u)
    = g(\bar p,\bar x)+t_{k}\nabla g(\bar p,\bar x)(p,u)+t_{k}^{2}\zeta_{k},
    \qquad \zeta_{k}\to d^{2}g(\bar p,\bar x)(p,u),
\end{equation}
Combining these expansions and absorbing the \(o(t_k^2)\) remainder into
the quadratic term, we obtain
\begin{equation}\label{eq:Taylor-g-1}
    g(\bar p+t_{k}p,\bar x+t_{k}u+t_{k}^{2}w_{k})
    = g(\bar p,\bar x)+t_{k}\nabla g(\bar p,\bar x)(p,u)
      +t_{k}^{2}\tilde z_{k},
\end{equation}
where
\(\tilde z_{k}:=\nabla_{x}g(\bar p+t_{k}p,\bar x+t_{k}u)\,w_{k}+\zeta_{k}+o(1)\to z\).

\emph{(\(\subset\)).} Let \(w\in\DD H(\bar p,\bar x,u)(p)\). By definition,
there exist \(t_{k}\downarrow 0\) and \(w_{k}\to w\) such that
\(\bar x+t_{k}u+t_{k}^{2}w_{k}\in H(\bar p+t_{k}p)\) for all \(k\in\N\).
Thus
\(\bar x+t_{k}u+t_{k}^{2}w_{k}\in \Omega\) and
\(g(\bar p+t_{k}p,\bar x+t_{k}u+t_{k}^{2}w_{k})\in D\). The
first membership gives \(w\in T^{2}(\Omega,\bar x,u)\). The
second, together with \eqref{eq:Taylor-g-1}, gives
\(z\in T^{2}(D,\,g(\bar p,\bar x),\,\nabla g(\bar p,\bar x)(p,u))\)
after dividing the displacement by \(t_{k}^{2}\) and passing to the limit.
Hence \(w\in\mathcal{R}\), which proves the inclusion
\(\DD H(\bar p,\bar x,u)(p)\subset\mathcal{R}\).

\emph{(\(\supset\)).} Let \(w\in T^{2}(\Omega,\bar x,u)\) and
\(z\in T^{2}(D,g(\bar p,\bar x),\nabla g(\bar p,\bar x)(p,u))\). By second-order regularity, these memberships can be realized along any
prescribed sequence \(t_{k}\downarrow0\); hence, for every such sequence, there exist \((w_{k},z_{k})\to(w,z)\) such that, for
all \(k\in\N\),
\begin{equation}\label{eq:KD-realize}
    \bar x+t_{k}u+t_{k}^{2}w_{k}\in \Omega,
    \quad\text{and}\quad
    g(\bar p,\bar x)+t_{k}\nabla g(\bar p,\bar x)(p,u)+t_{k}^{2}z_{k}\in D.
\end{equation}
{Since \(\bar x+t_{k}u+t_{k}^{2}w_{k}\in \Omega\),
Robinson metric regularity of \(H\) along \(\Omega\) at
\((\bar p,\bar x)\) gives \(\alpha>0\) such that, for all large \(k\),}
\[
    \dist\bigl(\bar x+t_{k}u+t_{k}^{2}w_{k},\,H(\bar p+t_{k}p)\bigr)
    \leqslant\alpha\,
    \dist\bigl(g(\bar p+t_{k}p,\bar x+t_{k}u+t_{k}^{2}w_{k}),\,D\bigr).
\]
By \eqref{eq:KD-realize} and \eqref{eq:Taylor-g-1}, the
right-hand side is at most
\(\alpha t_{k}^{2}\|\tilde z_{k}-z_{k}\|\).
Set \(\varepsilon_k:=t_k^3\) and choose
\(\bar x_k\in H(\bar p+t_kp)\) with
\[
    \|\bar x+t_ku+t_k^2w_k-\bar x_k\|
    \leqslant
    \dist\bigl(\bar x+t_ku+t_k^2w_k,H(\bar p+t_kp)\bigr)+\varepsilon_k.
\]
{Combined with the Robinson bound above, this gives}
\[
    \|\bar x+t_{k}u+t_{k}^{2}w_{k}-\bar x_{k}\|
    \leqslant\alpha t_{k}^{2}\|\tilde z_{k}-z_{k}\|+t_k^3.
\]
{Set \(\tilde w_{k}:=(\bar x_{k}-\bar x-t_{k}u)/t_{k}^{2}\). Then}
\(\|\tilde w_{k}-w_{k}\|\leqslant\alpha\|\tilde z_{k}-z_{k}\|+t_k\to 0\),
so \(\tilde w_{k}\to w\). Since
\(\bar x_{k}=\bar x+t_{k}u+t_{k}^{2}\tilde w_{k}\in H(\bar p+t_{k}p)\) for
the arbitrary sequence \(t_{k}\downarrow 0\), {the limit \(w\)
lies in \(\DD_{\ell}H(\bar p,\bar x,u)(p)\). Hence
\(\mathcal R\subset\DD_{\ell}H(\bar p,\bar x,u)(p)\). Together with the
already proved inclusion \(\DD H(\bar p,\bar x,u)(p)\subset\mathcal R\) and
the universal inclusion \(\DD_{\ell}H\subset\DD H\), this gives}
\(\DD_{\ell}H(\bar p,\bar x,u)(p)=\DD H(\bar p,\bar x,u)(p)=\mathcal R\),
which is exactly second-order semi-derivability of \(H\) at
\((\bar p,\bar x)\) in direction \((p,u)\) and \eqref{eq:DDH-str}.
\end{proof}

\subsection[Semi-derivatives of F]{\texorpdfstring{{Semi-derivatives of \(\mathcal F\)}}{Semi-derivatives of F}}\label{subsec:DF-struct}

\begin{lemma}{First-order objective-aware boundedness}{blowup-1}
Assume \((A_2)\) and \((A_5)\). Let \(t_k\downarrow0\) and
\(x_k\in H(\bar p+t_kp)\) satisfy
\(x_k\to\bar x\in\Delta_0(\bar p,\bar y)\). If
\[
 \left\{\frac{f(\bar p+t_kp,x_k)-\bar y}{t_k}\right\}
\]
is bounded, then \(\{(x_k-\bar x)/t_k\}\) is bounded.
\end{lemma}

\begin{proof}
Suppose \(\{(x_k-\bar x)/t_k\}\) is unbounded. Along a subsequence,
\(\|x_k-\bar x\|/t_k\to\infty\); set \(r_k:=\|x_k-\bar x\|\), so
\(t_k/r_k\to0\) and \(r_k\to0\) (since \(x_k\to\bar x\)). Passing to a
further subsequence, \((x_k-\bar x)/r_k\to\tilde u\) with
\(\|\tilde u\|=1\). Since \(x_k\in\Omega\), we get
\(\tilde u\in T(\Omega,\bar x)\). Dividing the \(C^1\) expansion of \(g\)
by \(r_k\) and using \(t_k/r_k\to0\) gives
\(\bigl(g(\bar p+t_kp,x_k)-g(\bar p,\bar x)\bigr)/r_k\to
\nabla_xg(\bar p,\bar x)\tilde u\); combined with
\(g(\bar p+t_kp,x_k)\in D\), this places
\(\nabla_xg(\bar p,\bar x)\tilde u\in T(D,g(\bar p,\bar x))\). The
same expansion for \(f\), divided by \(r_k\) and rewritten as
\(\bigl(f(\bar p+t_kp,x_k)-\bar y\bigr)/r_k
=(t_k/r_k)\cdot\bigl(f(\bar p+t_kp,x_k)-\bar y\bigr)/t_k\to0\)
by the assumed boundedness, yields
\(\nabla_xf(\bar p,\bar x)\tilde u=0\in-C\). Hence
\(\tilde u\in R_1(\bar p,\bar x)\setminus\{0\}\), contradicting \((A_5)\).
\end{proof}

\begin{lemma}{Second-order objective-aware boundedness}{blowup-2}
Assume \((A_3)\) and \((A_6)\). Let \(t_k\downarrow0\) and
\(x_k\in H(\bar p+t_kp)\) satisfy
\(x_k\to\bar x\in\Delta_0(\bar p,\bar y)\) and
\((x_k-\bar x)/t_k\to u\in\Delta_1(\bar p,\bar x,p,v)\). If
\[
 \eta_k:=\frac{f(\bar p+t_kp,x_k)-\bar y-t_kv}{t_k^2}
\]
is bounded, then
\[
 w_k:=\frac{x_k-\bar x-t_ku}{t_k^2}
\]
is bounded.
\end{lemma}

\begin{proof}
Suppose \(\{w_k\}\) is unbounded. Along a subsequence,
\(\|w_k\|\to\infty\); set \(s_k:=t_k\|w_k\|\). Then
\(t_k/s_k=1/\|w_k\|\to0\), and
\(s_k=\|x_k-\bar x-t_ku\|/t_k\to0\) since
\((x_k-\bar x)/t_k\to u\). Passing to a further subsequence,
\(w_k/\|w_k\|\to\tilde w\) with \(\|\tilde w\|=1\). Since
\((x_k-\bar x-t_ku)/(t_ks_k)=w_k/\|w_k\|\to\tilde w\) and
\(x_k\in\Omega\), the definition of \(T''\) gives
\(\tilde w\in T''(\Omega,\bar x,u)\). The second-order expansion of
\(g\) supplied by \((A_3)\), divided by \(t_ks_k\) and combined with
\(g(\bar p+t_kp,x_k)\in D\), gives
\(\nabla_xg(\bar p,\bar x)\tilde w\in
T''(D,g(\bar p,\bar x),\nabla g(\bar p,\bar x)(p,u))\).
The second-order expansion of \(f\), divided by \(t_ks_k\) and rewritten
as \(\bigl(f(\bar p+t_kp,x_k)-\bar y-t_kv\bigr)/(t_ks_k)
=(t_k/s_k)\,\eta_k\to0\) by boundedness of \(\eta_k\), yields
\(\nabla_xf(\bar p,\bar x)\tilde w=0\in-C\). Hence
\(\tilde w\in R_2(\bar p,\bar x,u,p)\setminus\{0\}\), contradicting \((A_6)\).
\end{proof}

The first-order formula below makes the direction set \(\gph D\mathcal F(\bar p,\bar y)\) of Proposition~\ref{prop:DDF-str} explicit.

\begin{proposition}{First-order semi-derivative of \(\mathcal F\) in the structured setting}{DF-str-first-order}
Let \((\bar p,\bar y)\in\gph\mathcal F\) and \(p\in P\). Assume
\((A_1)\) and conditions \((A_2)\) and \((A_5)\) of
Subsection~\ref{subsec:sec4-setup}.
Then \(\mathcal F\) is
semi-derivable at \((\bar p,\bar y)\) in direction \(p\), and

\[
    D\mathcal F(\bar p,\bar y)(p)
    =\bigcup_{\substack{\bar x\in\Delta_{0}(\bar p,\bar y)\\
                        u\in DH(\bar p,\bar x)(p)}}
      \bigl\{\nabla f(\bar p,\bar x)(p,u)\bigr\}.
\]
\end{proposition}

\begin{proof}
Let \(a\in D\mathcal F(\bar p,\bar y)(p)\). Choose
\(t_{k}\downarrow0\), \(a_{k}\to a\), and
\(x_{k}\in H(\bar p+t_{k}p)\) such that
\(f(\bar p+t_{k}p,x_{k})=\bar y+t_{k}a_{k}\). The objective identity
and \((A_1)\) show that \(\{x_k\}\) is bounded. After passing to a
subsequence, let \(x_k\to\bar x\). Closedness of \(\Omega,D\) and
continuity of \(g\) give \(\bar x\in H(\bar p)\). Continuity of \(f\)
gives \(f(\bar p,\bar x)=\bar y\), so
\(\bar x\in\Delta_{0}(\bar p,\bar y)\).
The objective identity gives
\((f(\bar p+t_kp,x_k)-\bar y)/t_k=a_k\to a\), which is bounded.
Lemma~\ref{lem:blowup-1} then yields boundedness of
\(\{(x_k-\bar x)/t_k\}\).

Thus, along a subsequence,
\((x_k-\bar x)/t_k\to u\in DH(\bar p,\bar x)(p)\), and the \(C^{1}\)
expansion of \(f\) gives
\(a=\nabla f(\bar p,\bar x)(p,u)\). This proves the upper inclusion.

Conversely, fix \(\bar x\in\Delta_{0}(\bar p,\bar y)\) and
\(u\in DH(\bar p,\bar x)(p)\), and let \(t_k\downarrow0\) be arbitrary.
The semi-derivability of \(H\) provides \(u_k\to u\) such that
\(\bar x+t_ku_k\in H(\bar p+t_kp)\). Hence
\[
 \frac{f(\bar p+t_kp,\bar x+t_ku_k)-\bar y}{t_k}
 \longrightarrow \nabla f(\bar p,\bar x)(p,u),
\]
so every element of the displayed union belongs to
\(D_{\ell}\mathcal F(\bar p,\bar y)(p)\). The universal inclusion
\(D_{\ell}\mathcal F\subset D\mathcal F\) completes the proof.
\end{proof}

\begin{proposition}{Second-order semi-derivative of \(\mathcal F\) in the structured setting}{DDF-str}
Let \((\bar p,\bar y)\in\gph\mathcal F\) and \((p,v)\in\gph D\mathcal F(\bar p,\bar y)\). Assume conditions \((A_1)\)--\((A_6)\) of Subsection~\ref{subsec:sec4-setup}. Then \(\mathcal F\) is second-order semi-derivable at \((\bar p,\bar y)\) in direction \((p,v)\), and
\begin{equation}\label{eq:DDF-str}
    \DD\mathcal F(\bar p,\bar y,v)(p)
    =\bigcup_{\substack{\bar x\in\Delta_{0}(\bar p,\bar y)\\ u\in\Delta_{1}(\bar p,\bar x,p,v)\\
                                                w\in\DD H(\bar p,\bar x,u)(p)}}
      \bigl\{\nabla_{x}f(\bar p,\bar x)\,w+d^{2}f(\bar p,\bar x)(p,u)\bigr\},
\end{equation}
where \(\DD H(\bar p,\bar x,u)(p)\) is given by \eqref{eq:DDH-str}.
\end{proposition}

\begin{proof}
Denote by \(\mathcal U\) the right-hand side of \eqref{eq:DDF-str}. For every \(\bar x\in\Delta_{0}(\bar p,\bar y)\) and \(u\in\Delta_{1}(\bar p,\bar x,p,v)\), the first-order compatibility condition \eqref{eq:DDH-str-fo-compat} holds; hence Proposition~\ref{prop:DDH-str}, applied under \((A_2)\), \((A_4)\) and the regularity clauses of \((A_3)\), gives the second-order semi-derivability of \(H\) at \((\bar p,\bar x)\) in direction \((p,u)\).

For every \(\bar x\in\Delta_{0}(\bar p,\bar y)\), \(u\in\Delta_{1}(\bar p,\bar x,p,v)\), and \(w\in\DD H(\bar p,\bar x,u)(p)\), Proposition~\ref{prop:DDH-str} yields, for every
\(t_{k}\downarrow0\), a sequence \(w_{k}\to w\) with
\(\bar x+t_{k}u+t_{k}^{2}w_{k}\in H(\bar p+t_{k}p)\). The \(C^{1}\)
mean-value theorem together with second-order semi-derivability of \(f\)
then yields
\(\nabla_{x}f(\bar p,\bar x)w+d^{2}f(\bar p,\bar x)(p,u)\in
\DD_{\ell}\mathcal F(\bar p,\bar y,v)(p)\); cf.\ \cite[Thm.~4.2]{Bao2025JOTA}. For the reverse inclusion, take
\(y\in\DD\mathcal F(\bar p,\bar y,v)(p)\). Then there are
\(t_{k}\downarrow0\), \(\eta_{k}\to y\) and \(x_{k}\in H(\bar p+t_{k}p)\) with
\begin{equation}\label{eq:DDF-str-value}
   f(\bar p+t_{k}p,x_{k})=\bar y+t_{k}v+t_{k}^{2}\eta_{k}.
\end{equation}
The objective identity \eqref{eq:DDF-str-value} and \((A_1)\) show
that \(\{x_k\}\) is bounded. After passing to a subsequence, let
\(x_k\to\bar x\). Closedness of \(\Omega,D\) and continuity of \(g\)
imply \(\bar x\in H(\bar p)\). Passing to the limit in
\eqref{eq:DDF-str-value} by continuity of \(f\) gives
\(f(\bar p,\bar x)=\bar y\), hence
\(\bar x\in\Delta_{0}(\bar p,\bar y)\).

\emph{First-order extraction.} The objective identity gives
\((f(\bar p+t_kp,x_k)-\bar y)/t_k=v+t_k\eta_k\), which is bounded.
Lemma~\ref{lem:blowup-1} then yields boundedness of
\(\{(x_k-\bar x)/t_k\}\). Along a subsequence,
\((x_k-\bar x)/t_k\to u\). Feasibility gives
\(u\in DH(\bar p,\bar x)(p)\), and the \(C^1\) expansion of \(f\) gives
\(\nabla f(\bar p,\bar x)(p,u)=v\). Hence
\(u\in\Delta_1(\bar p,\bar x,p,v)\). In particular, if this
\(\Delta_1\)-set is empty at a fixed \(\bar x\), no realizing sequence
can converge to that branch; otherwise the extraction above would produce
such a vector \(u\).

\emph{Second-order extraction.} The quotient
\(\eta_k=(f(\bar p+t_kp,x_k)-\bar y-t_kv)/t_k^2\) is bounded.
Lemma~\ref{lem:blowup-2}, applied with \((\bar x,u)\) chosen above,
yields boundedness of
\(w_k:=(x_k-\bar x-t_ku)/t_k^2\). Along a subsequence,
\(w_k\to w\in\DD H(\bar p,\bar x,u)(p)\). Passing to the limit in
\eqref{eq:DDF-str-value} gives
\(y=\nabla_xf(\bar p,\bar x)w+d^2f(\bar p,\bar x)(p,u)\in\mathcal U\).

Hence \(\DD\mathcal F=\DD_{\ell}\mathcal F=\mathcal U\).
\end{proof}

\subsection[Second-order semi-derivative of Phi]{\texorpdfstring{{Second-order semi-derivative of \(\Phi\)}}{Second-order semi-derivative of Phi}}\label{subsec:DDPhi-struct}

\begin{lemma}{Domination in compact sections}{compact-domination}
Let \(p\in P\) and \(y\in\mathcal F(p)\). If the section
\(\mathcal F(p)\cap(y-C)\) is nonempty and compact, then there exist
\(\xi\in S(p)\) and \(c\in C\) with \(y=f(p,\xi)+c\). In particular
\(f(p,\xi)\in\Phi(p)\).
\end{lemma}

\begin{proof}
Set \(A:=\mathcal F(p)\cap(y-C)\) and choose
\(\lambda\in\Int C^*\). Compactness gives a minimizer \(z\) of
\(\langle\lambda,\cdot\rangle\) on \(A\). Since \(\lambda\) is strictly
positive on \(C\setminus\{0\}\), this minimizer belongs to \(\Min_C A\),
and \(y-z\in C\). If \(w\in\mathcal F(p)\cap(z-C)\), then \(w\in A\),
so minimality gives \(w=z\). Hence \(z\in\Phi(p)\); writing
\(z=f(p,\xi)\) yields \(\xi\in S(p)\) and proves the claim.
\end{proof}

\begin{remark}{Local validity of the compact-section hypothesis}{compact-domination-local}
Under \((A_1)\), for \((p,y)\) near \((\bar p,\bar y)\), the set
\[
 \mathcal L(p,y):=\{x\in H(p):f(p,x)\in y-C\}
\]
is bounded. Closedness of \(\Omega,D,C\) and continuity of \(f,g\)
make \(\mathcal L(p,y)\) closed, hence compact. Its image is
\(\mathcal F(p)\cap(y-C)\), so
Lemma~\ref{lem:compact-domination} applies.
\end{remark}

\begin{theorem}{Second-order semi-derivative of \(\Phi\) in the structured setting}{DDPhi-str}
Let \(\bar p\in P\) and let \(\bar y\) be a uniformly Henig efficient
point of \(\mathcal F\) at \(\bar p\), with \((p,v)\in\gph D\Phi(\bar p,\bar y)\).
Assume conditions \((A_1)\)--\((A_6)\) of Subsection~\ref{subsec:sec4-setup}; see
Table~\ref{tab:sec3-hypotheses}.
Then \(\Phi\) is second-order semi-derivable at \((\bar p,\bar y)\) in
direction \((p,v)\), and
\begin{equation}\label{eq:DDPhi-str}
    \DD\Phi(\bar p,\bar y,v)(p)
    = \Min\nolimits_{C}\!\!\!\bigcup_{\substack{\bar x\in\Delta_{0}(\bar p,\bar y)\\ u\in\Delta_{1}(\bar p,\bar x,p,v)\\
                                                w\in\DD H(\bar p,\bar x,u)(p)}}
      \bigl\{\nabla_{x}f(\bar p,\bar x)\,w+d^{2}f(\bar p,\bar x)(p,u)\bigr\},
\end{equation}
where \(\DD H(\bar p,\bar x,u)(p)\) is given by \eqref{eq:DDH-str}.
\end{theorem}

\begin{proof}
By Proposition~\ref{prop:DDF-str}, \(\mathcal F\) is second-order semi-derivable at \((\bar p,\bar y)\) in direction \((p,v)\), and \(\DD\mathcal F(\bar p,\bar y,v)(p)=\mathcal U\), where \(\mathcal U\) denotes the set on the right-hand side of \eqref{eq:DDPhi-str} before taking \(\Min_{C}\). Under uniform Henig efficiency, \cite[Lem.~3.1]{Bao2025JOTA} gives both
\[
   D\Phi(\bar p,\bar y)(p)\subset \Min_{C}D\mathcal F(\bar p,\bar y)(p)
   \quad\text{and}\quad
   \DD\Phi(\bar p,\bar y,v)(p)\subset \Min_{C}\DD\mathcal F(\bar p,\bar y,v)(p);
\]
the second inclusion uses second-order semi-derivability of \(\mathcal F\) from Proposition~\ref{prop:DDF-str}. Since \((p,v)\in\gph D\Phi(\bar p,\bar y)\), the first inclusion yields \(v\in\Min_{C}D\mathcal F(\bar p,\bar y)(p)\), a fact used in Step~1 below. It remains to prove the reverse of the second inclusion.

Take \(y\in\Min_{C}\DD\mathcal F(\bar p,\bar y,v)(p)\) and fix an arbitrary sequence \(t_{k}\downarrow0\). Since Proposition~\ref{prop:DDF-str} gives
\(\DD\mathcal F(\bar p,\bar y,v)(p)=\DD_{\ell}\mathcal F(\bar p,\bar y,v)(p)\),
we can choose \(y_{k}\to y\) such that
\[
   \bar y_{k}:=\bar y+t_{k}v+t_{k}^{2}y_{k}
      \in\mathcal F(\bar p+t_{k}p).
\]
For all large \(k\), Remark~\ref{rem:compact-domination-local} shows
that the section
\(\mathcal F(\bar p+t_kp)\cap(\bar y_k-C)\) is compact. Hence
Lemma~\ref{lem:compact-domination} applies to
\(\bar y_{k}\in\mathcal F(\bar p+t_{k}p)\), yielding
\(\bar z_{k}\in\Phi(\bar p+t_{k}p)\), \(c_{k}\in C\), and
\(\xi_{k}\in S(\bar p+t_{k}p)\) such that
\[
   \bar y_{k}=\bar z_{k}+c_{k},\qquad
   \bar z_{k}=f(\bar p+t_{k}p,\xi_{k}).
\]

The remaining argument has two steps: Step~1 controls \(c_{k}\) at first and second order; Step~2 uses \(C\)-minimality of \(y\) to identify the limit.

\emph{Step 1: control of the correction \(c_{k}\).}
We prove \(y\in\DD_{\ell}\Phi(\bar p,\bar y,v)(p)\). By the subsequence principle it suffices to extract, from an arbitrary subsequence of the given \(\{t_{k}\}\), a further subsequence along which
\((\bar z_{k}-\bar y-t_{k}v)/t_{k}^{2}\to y\); recall \(\bar z_{k}\in\Phi(\bar p+t_{k}p)\). Fix such a subsequence.

\emph{\textup{(1a)} Base point and vanishing of \(c_{k}\).}
Since \(\bar z_k\in\bar y_k-C\), \(\bar y_k\to\bar y\), and \(\bar z_k=f(\bar p+t_kp,\xi_k)\), condition \((A_1)\) shows that \(\{\xi_k\}\) is bounded. After passing to a further subsequence, let \(\xi_k\to\bar x\). Closedness of \(\Omega,D\) and continuity of \(g\) give \(\bar x\in H(\bar p)\); continuity of \(f\) then gives \(\bar z_{k}\to f(\bar p,\bar x)\). Hence
\[
   c_{k}=\bar y_{k}-\bar z_{k}\to\bar y-f(\bar p,\bar x)=:c\in C,
\]
using closedness of \(C\) and \(c_{k}\in C\). As \(\bar x\in H(\bar p)\), \(f(\bar p,\bar x)\in\mathcal F(\bar p)\) and \(f(\bar p,\bar x)-\bar y=-c\in-C\); the efficiency of \(\bar y\), i.e.\ \((\mathcal F(\bar p)-\bar y)\cap(-C\setminus\{0\})=\emptyset\), forces \(c=0\). Thus \(f(\bar p,\bar x)=\bar y\), so \(\bar x\in\Delta_{0}(\bar p,\bar y)\), \(\bar z_{k}\to\bar y\), and \(c_{k}\to0\).

\emph{\textup{(1b)} Upgrade to \(c_{k}=o(t_{k})\).}
By condition \((A_2)\), \(f\) is locally Lipschitz near \((\bar p,\bar x)\) with some constant \(L\); the same holds for \(g\). Suppose \(\{c_{k}/t_{k}\}\) is unbounded and pass to a subsequence with \(\|c_{k}\|/t_{k}\to\infty\). Since \(\bar y=f(\bar p,\bar x)\) by \textup{(1a)}, the triangle inequality and Lipschitz continuity of \(f\) give
\[
   \|c_{k}\|
   =\|\bar y_{k}-\bar z_{k}\|
   \leqslant\|\bar y_{k}-\bar y\|+\|\bar z_{k}-\bar y\|
   \leqslant\bigl(t_{k}\|v\|+t_{k}^{2}\|y_{k}\|\bigr)+L\bigl(t_{k}\|p\|+\|\xi_{k}-\bar x\|\bigr).
\]
Dividing by \(t_{k}\) and letting \(k\to\infty\) forces \(\|\xi_{k}-\bar x\|/t_{k}\to\infty\); set \(r_{k}:=\|\xi_{k}-\bar x\|\), so \(t_{k}/r_{k}\to 0\). Dividing the same bound by \(r_{k}\) yields \(\|c_{k}\|/r_{k}\) bounded; pass to a further subsequence with \(\|c_{k}\|/r_{k}\to\mu\geqslant 0\) and \(c_{k}/\|c_{k}\|\to\tilde c\in C\).

Let \(\eta:=\lim(\xi_{k}-\bar x)/r_{k}\), a unit vector; \(\xi_k\in\Omega\) and closedness of \(\Omega\) give \(\eta\in T(\Omega,\bar x)\). By \((A_2)\) and \(t_{k}/r_{k}\to 0\), the \(C^{1}\) expansion of \(f\) at \((\bar p,\bar x)\), using \(\bar y=f(\bar p,\bar x)\), yields
\[
   \frac{\bar z_{k}-\bar y}{r_{k}}
   =\frac{f(\bar p+t_{k}p,\xi_{k})-f(\bar p,\bar x)}{r_{k}}
   \;\longrightarrow\;\nabla_{x}f(\bar p,\bar x)\eta;
\]
the same expansion for \(g\), together with \(g(\bar p+t_{k}p,\xi_{k})\in D\), \(g(\bar p,\bar x)\in D\), and closedness of \(D\), gives \(\nabla_{x}g(\bar p,\bar x)\eta\in T(D,g(\bar p,\bar x))\). On the other hand, \((\bar y_{k}-\bar y)/r_{k}=(t_{k}/r_{k})(v+t_{k}y_{k})\to 0\) and \(c_{k}/r_{k}\to\mu\tilde c\), so
\[
   \frac{\bar z_{k}-\bar y}{r_{k}}
   =\frac{\bar y_{k}-\bar y}{r_{k}}-\frac{c_{k}}{r_{k}}
   \;\longrightarrow\;-\mu\tilde c.
\]
Comparing the two limits of \((\bar z_{k}-\bar y)/r_{k}\) gives
\[
   \nabla_{x}f(\bar p,\bar x)\eta=-\mu\tilde c\in -C,
\]
contradicting \((A_5)\). Hence \(\{c_{k}/t_{k}\}\) is bounded; pass to a subsequence with \(c_{k}/t_{k}\to\gamma\in C\). Then
\[
   \frac{\bar z_{k}-\bar y}{t_{k}}\to v-\gamma
      \in D\mathcal F(\bar p,\bar y)(p),
\]
and the relation \(v\in\Min_{C}D\mathcal F(\bar p,\bar y)(p)\) from the opening forces \(\gamma=0\). Thus \(c_{k}=o(t_{k})\).

\emph{\textup{(1c)} Upgrade to \(c_{k}=O(t_{k}^{2})\).}
Applying the first-order extraction from the proof of Proposition~\ref{prop:DDF-str} to \(\xi_{k}\), we may assume \((\xi_{k}-\bar x)/t_{k}\to u\in\Delta_{1}\). Set
\[
   w_{k}:=\frac{\xi_{k}-\bar x-t_{k}u}{t_{k}^{2}},
   \qquad
   \zeta_{k}:=\frac{\bar z_{k}-\bar y-t_{k}v}{t_{k}^{2}}
      =y_{k}-\frac{c_{k}}{t_{k}^{2}}.
\]
Suppose \(\{\zeta_{k}\}\) is unbounded. Along a subsequence, \(\|c_{k}\|/t_{k}^{2}\to\infty\) and \(\zeta_{k}/\|\zeta_{k}\|\to-\hat c\in-C\) with \(\|\hat c\|=1\). Fixed-direction second-order semi-derivability of \(f\) yields the mean-value expansion
\[
   \zeta_{k}=\nabla_{x}f(\bar p,\bar x)w_{k}+\theta_{k}+o(\|w_{k}\|),
\]
where \(\theta_{k}=d^{2}f(\bar p,\bar x)(p,u)+o(1)\) is bounded and the \(o(\|w_{k}\|)\) remainder comes from continuity of \(\nabla_{x}f\) along the mean-value segment (using \(t_{k}w_{k}\to0\)). The expansion forces \(\|w_{k}\|\to\infty\).

Let \(\tilde w:=\lim w_{k}/\|w_{k}\|\). Since \(\|w_{k}\|\to\infty\), the tangency argument in the proof of Lemma~\ref{lem:blowup-2}---which uses only \((A_3)\), not the boundedness of the value quotient---gives the feasibility conditions in \((A_6)\) for \(\tilde w\). After a further subsequence, \(\|\zeta_{k}\|/\|w_{k}\|\to\rho\geqslant0\), and dividing the mean-value expansion by \(\|w_{k}\|\) yields
\[
   \nabla_{x}f(\bar p,\bar x)\tilde w=-\rho\hat c\in -C,
\]
contradicting \((A_6)\). Hence \(\{\zeta_{k}\}\) is bounded, and therefore \(c_{k}=O(t_{k}^{2})\).

\emph{Step 2: minimality gives the required limit.}
\emph{\textup{(2a)} Identification of the accumulation point.}
Apply the second-order extraction from the proof of Proposition~\ref{prop:DDF-str} to \(\xi_{k}\), with value quotient \(\zeta_{k}\). Along a further subsequence,
\[
   \zeta_{k}\to
   \zeta=\nabla_{x}f(\bar p,\bar x)w+d^{2}f(\bar p,\bar x)(p,u)
      \in\DD\mathcal F(\bar p,\bar y,v)(p)
\]
for some \(w\in\DD H(\bar p,\bar x,u)(p)\). Since \(c_{k}/t_{k}^{2}=y_{k}-\zeta_{k}\), we have \(y-\zeta\in C\), and \(C\)-minimality of \(y\in\Min_{C}\DD\mathcal F(\bar p,\bar y,v)(p)\) forces \(\zeta=y\). Along this subsequence,
\[
   \zeta_{k}=\frac{\bar z_{k}-\bar y-t_{k}v}{t_{k}^{2}}\to y,
   \qquad
   \frac{c_{k}}{t_{k}^{2}}=y_{k}-\zeta_{k}\to0.
\]

\emph{\textup{(2b)} Passing to the whole sequence and conclusion.}
The argument above applies to any subsequence of the initial \(\{t_{k}\}\) and forces every accumulation point of \(\{\zeta_{k}\}\) to equal the fixed \(y\); hence \(y\) is the unique accumulation point, and the subsequence principle gives \(\zeta_{k}\to y\) for the whole given sequence:
\[
   \frac{\bar z_{k}-\bar y-t_{k}v}{t_{k}^{2}}\to y,
   \qquad
   \bar z_{k}\in\Phi(\bar p+t_{k}p).
\]
Since \(t_{k}\downarrow0\) was the arbitrary sequence fixed at the outset, \(y\in\DD_{\ell}\Phi(\bar p,\bar y,v)(p)\), hence
\[
   \Min_{C}\DD\mathcal F(\bar p,\bar y,v)(p)\subset\DD_{\ell}\Phi(\bar p,\bar y,v)(p).
\]
Combining with \(\DD\Phi\subset\Min_{C}\DD\mathcal F\) from \cite[Lem.~3.1]{Bao2025JOTA} and the universal \(\DD_{\ell}\Phi\subset\DD\Phi\) yields
\(\DD_{\ell}\Phi=\DD\Phi=\Min_{C}\DD\mathcal F(\bar p,\bar y,v)(p)\). Thus \(\Phi\) is second-order semi-derivable at \((\bar p,\bar y)\) in direction \((p,v)\), and substituting \(\DD\mathcal F=\mathcal U\) gives~\eqref{eq:DDPhi-str}.
\end{proof}

\section[Second-order sensitivity of S]{\texorpdfstring{{Second-order sensitivity of the efficient solution map \(S\)}}{Second-order sensitivity of S}}\label{sec:effsol}

This section transfers second-order sensitivity from
\(\Phi\) to \(S\). Fix \((\bar p,\bar x)\in\gph S\), set
\(\bar y=f(\bar p,\bar x)\), take
\(u\in DS(\bar p,\bar x)(p)\), and define
\begin{equation}\label{eq:vbar}
    \bar v := \nabla f(\bar p,\bar x)(p,u)\in Y.
\end{equation}
Theorem~\ref{thm:DDS-semideriv} gives the abstract formula;
Section~\ref{sec:paramsys} specializes it to structured \(H\), and
Proposition~\ref{prop:FMR-implies-VDB} provides one data-level
certificate for \textup{(VDB)}.

\subsection[First-order compatibility]{\texorpdfstring{{First-order compatibility}}{First-order compatibility}}\label{subsec:compat1}

The definition of \(S\) gives the following first-order
compatibility.

\begin{proposition}{First-order compatibility for \(S\)}{compat1}
Let \((\bar p,\bar x)\in\gph S\), \(\bar y=f(\bar p,\bar x)\in\Phi(\bar p)\),
and \(p\in P\). Assume that \(f\) is continuously differentiable at
\((\bar p,\bar x)\). Then, for every \(u\in DS(\bar p,\bar x)(p)\),
\begin{equation}\label{eq:compat1}
    u\in DH(\bar p,\bar x)(p)\qquad\text{and}\qquad
    \nabla f(\bar p,\bar x)(p,u)\in D\Phi(\bar p,\bar y)(p).
\end{equation}
\end{proposition}

\begin{proof}
By definition of the first-order upper Dini derivative of \(S\), there
exist \(t_{k}\downarrow 0\) and \(u_{k}\to u\) such that
\(\bar x+t_{k}u_{k}\in S(\bar p+t_{k}p)\) for all \(k\in\N\). By the
definition of \(S\) in \eqref{eq:Sdef}, this is equivalent to
\begin{align}
    \bar x+t_{k}u_{k} &\in H(\bar p+t_{k}p),\label{eq:c1-H}\\
    f(\bar p+t_{k}p,\bar x+t_{k}u_{k}) &\in \Phi(\bar p+t_{k}p).\label{eq:c1-Phi}
\end{align}
From \eqref{eq:c1-H}, \((\bar x+t_{k}u_{k}-\bar x)/t_{k}=u_{k}\to u\)
witnesses \(u\in DH(\bar p,\bar x)(p)\).

Set \(y_{k}:=f(\bar p+t_{k}p,\bar x+t_{k}u_{k})\). Continuous
differentiability of \(f\) gives the first-order expansion
\[
    y_{k}=\bar y+t_{k}\nabla f(\bar p,\bar x)(p,u_{k})+o(t_{k})
         =\bar y+t_{k}\nabla f(\bar p,\bar x)(p,u)+o(t_{k}).
\]
Hence \((y_{k}-\bar y)/t_{k}\to\nabla f(\bar p,\bar x)(p,u)\). Since
\(y_{k}\in\Phi(\bar p+t_{k}p)\) by \eqref{eq:c1-Phi},
\(\nabla f(\bar p,\bar x)(p,u)\in D\Phi(\bar p,\bar y)(p)\), which is the
second part of \eqref{eq:compat1}.
\end{proof}

\subsection[Second-order outer estimate]{\texorpdfstring{{Second-order outer estimate}}{Second-order outer estimate}}\label{subsec:outer2}

The outer estimate is the necessary chain rule. Every
second-order direction of \(S\) is a second-order direction of \(H\), and
its image under the second-order expansion of \(f\) is a second-order
direction of \(\Phi\).

\begin{proposition}{Second-order outer estimate for \(S\)}{DDS-outer}
Let \((\bar p,\bar x)\in\gph S\), \(\bar y=f(\bar p,\bar x)\in\Phi(\bar p)\),
and \((p,u)\in P\times X\) with \(u\in DS(\bar p,\bar x)(p)\). Define
\(\bar v\) by \eqref{eq:vbar}. Assume that \(f\) is \(C^{1}\) and
second-order semi-derivable at \((\bar p,\bar x)\) in direction \((p,u)\).
Then, for every \(w\in\DD S(\bar p,\bar x,u)(p)\),
\begin{equation}\label{eq:DDS-outer}
    w\in\DD H(\bar p,\bar x,u)(p),
\end{equation}
and
\begin{equation}\label{eq:DDS-outer-image}
    \nabla_{x}f(\bar p,\bar x)\,w+d^{2}f(\bar p,\bar x)(p,u)
    \in \DD\Phi(\bar p,\bar y,\bar v)(p).
\end{equation}
\end{proposition}

\begin{proof}
Take \(w\in\DD S(\bar p,\bar x,u)(p)\), realized by
\(t_k\downarrow0\), \(w_k\to w\), and
\(x_k=\bar x+t_ku+t_k^2w_k\in S(\bar p+t_kp)\). Then
\(x_k\in H(\bar p+t_kp)\), so \eqref{eq:DDS-outer} holds. Put
\(y_k=f(\bar p+t_kp,x_k)\in\Phi(\bar p+t_kp)\). The mean-value formula
in \(x\), continuity of \(\nabla_xf\), and \(w_k\to w\) give
\[
 f(\bar p+t_kp,x_k)-f(\bar p+t_kp,\bar x+t_ku)
 =t_k^2\nabla_xf(\bar p,\bar x)w+o(t_k^2).
\]
Combining this with second-order semi-derivability of \(f\) yields
\(y_k=\bar y+t_k\bar v+t_k^2[\nabla_xf(\bar p,\bar x)w+
d^2f(\bar p,\bar x)(p,u)]+o(t_k^2)\). Since \(y_k\in\Phi(\bar p+t_kp)\),
the definition of \(\DD\Phi\) gives \eqref{eq:DDS-outer-image}.
\end{proof}

\subsection[Second-order semi-derivability of S]{\texorpdfstring{{Second-order semi-derivability of \(S\)}}{Second-order semi-derivability of S}}\label{subsec:semideriv2}

The outer estimate and the following converse argument give the full
second-order Dini formula. The converse is the only step that uses
\textup{(VDB)}.

{\small
\begin{remark}{Why the rate in \textup{(VDB)} must be linear}{VDB-linear-rate}
In the converse construction, semi-derivability of \(H\) and \(\Phi\)
yields a feasible \(x_k\) with
\(\dist(f(\bar p+t_kp,x_k),\Phi(\bar p+t_kp))=o(t_k^2)\).
Formula~\eqref{eq:DDS-formula} requires
\(\dist(x_k,S(\bar p+t_kp))=o(t_k^2)\). Linear \textup{(VDB)} supplies
this correction; an exponent-\(1/2\) bound gives only \(o(t_k)\), too
coarse for second-order realization.
\end{remark}%
}

\begin{theorem}{Second-order semi-derivability of \(S\)}{DDS-semideriv}
Let \((\bar p,\bar x)\in\gph S\), \(\bar y=f(\bar p,\bar x)\in\Phi(\bar p)\),
and \((p,u)\in P\times X\) with \(u\in DS(\bar p,\bar x)(p)\). Define
\(\bar v\) by \eqref{eq:vbar}. Assume:
\begin{enumerate}[label=\textup{(\roman*)}]
\item \(f\) is \(C^{1}\) and second-order semi-derivable at
      \((\bar p,\bar x)\) in direction \((p,u)\);
\item \(H\) is second-order semi-derivable at \((\bar p,\bar x)\) in
      direction \((p,u)\);
\item \(\Phi\) is second-order semi-derivable at \((\bar p,\bar y)\) in
      direction \((p,\bar v)\);
\item \textup{(VDB)} holds at \((\bar p,\bar x)\).
\end{enumerate}
Then \(S\) is second-order semi-derivable at \((\bar p,\bar x)\) in
direction \((p,u)\), and
\begin{equation}\label{eq:DDS-formula}
    \DD S(\bar p,\bar x,u)(p)
    = \left\{w\in\DD H(\bar p,\bar x,u)(p)\,:\,
       \begin{array}{l}
          \nabla_{x}f(\bar p,\bar x)w+d^{2}f(\bar p,\bar x)(p,u)\\
          \quad\in \DD\Phi(\bar p,\bar y,\bar v)(p)
       \end{array}\right\}.
\end{equation}
\end{theorem}

\begin{proof}
Let \(\mathcal W\) denote the right-hand side of \eqref{eq:DDS-formula}.
Proposition~\ref{prop:DDS-outer} gives
\(\DD S(\bar p,\bar x,u)(p)\subset\mathcal W\). For the reverse
inclusion, fix \(w\in\mathcal W\) and an arbitrary sequence
\(t_k\downarrow0\). The second-order semi-derivability of \(H\) gives
\(w_k\to w\) such that
\[
 x_k:=\bar x+t_ku+t_k^2w_k\in H(\bar p+t_kp).
\]
Set \(z:=d^2f(\bar p,\bar x)(p,u)\) and
\(\eta:=\nabla_xf(\bar p,\bar x)w+z\). Since
\(\eta\in\DD\Phi(\bar p,\bar y,\bar v)(p)\) and \(\Phi\) is
second-order semi-derivable, there are \(\xi_k\to\eta\) with
\[
 \tilde y_k:=\bar y+t_k\bar v+t_k^2\xi_k
 \in\Phi(\bar p+t_kp).
\]
The \(C^1\) mean-value theorem and second-order semi-derivability of \(f\)
give \(z_k\to z\) such that
\[
 f(\bar p+t_kp,x_k)
 =\bar y+t_k\bar v+t_k^2
   \bigl(\nabla_xf(\bar p,\bar x)w_k+z_k\bigr)+o(t_k^2).
\]
The coefficient in parentheses converges to \(\eta\). Hence
\[
 \dist\bigl(f(\bar p+t_kp,x_k),\Phi(\bar p+t_kp)\bigr)
 \leqslant\|f(\bar p+t_kp,x_k)-\tilde y_k\|=o(t_k^2).
\]
Let \(V,W,\kappa\) be the data in \textup{(VDB)}. For all large \(k\),
\((\bar p+t_kp,x_k)\in V\times W\), and therefore
\[
 \dist\bigl(x_k,S(\bar p+t_kp)\cap W\bigr)=o(t_k^2).
\]
This distance is finite, so the localized solution set is nonempty. Choose
\(\varepsilon_k=o(t_k^2)\) and
\(\tilde x_k\in S(\bar p+t_kp)\cap W\) with
\[
 \|\tilde x_k-x_k\|
 \leqslant\dist\bigl(x_k,S(\bar p+t_kp)\cap W\bigr)+\varepsilon_k.
\]
Writing
\(\tilde x_k=\bar x+t_ku+t_k^2\tilde w_k\), we obtain
\[
 \|\tilde w_k-w_k\|
 \leqslant\frac{\dist(x_k,S(\bar p+t_kp)\cap W)+\varepsilon_k}{t_k^2}
 \longrightarrow0.
\]
Thus \(\tilde w_k\to w\), and the arbitrariness of \(\{t_k\}\) gives
\(w\in\DD_\ell S(\bar p,\bar x,u)(p)\). Consequently,
\[
 \mathcal W\subset\DD_\ell S(\bar p,\bar x,u)(p)
 \subset\DD S(\bar p,\bar x,u)(p)\subset\mathcal W.
\]
Hence \(\DD_\ell S=\DD S=\mathcal W\), which proves the theorem.
\end{proof}

\begin{remark}{Directional use of \textup{(VDB)}}{VDB-directional}
For Theorem~\ref{thm:DDS-semideriv}, it is enough that \eqref{eq:VDB} hold with fixed \(W\) and \(\kappa\) at
\(p=\bar p+tp\) for every sufficiently small \(t>0\) and \(x\in H(\bar p+tp)\cap W\).
The neighborhood form is retained because it is direction-independent and easier to verify.
\end{remark}

The next example shows that the inner inclusion can fail
without \textup{(VDB)}.

\begin{example}{Failure of the inner estimate without \textup{(VDB)}}{DDS-counter}
Let \(X=Y=\R^{2}\), \(P=\R\), and \(C=\R^{2}_{+}\). Take
\[
    H(p)=\{x\in\R^{2}\,:\,x_{1}+x_{2}\leqslant 2+p\},\qquad
    f(p,x)=\bigl(\phi(p,x),\phi(p,x)\bigr),
\]
where \(\phi(p,x):=(x_{1}-p)^{2}+(x_{2}-p^{2})^{2}\). Take
\(\bar p=0\), \(\bar x=(0,0)\). Direct minimization of
\(f(p,\cdot)\) over \(H(p)\) yields the unique \(C\)-minimum
\(x^{\ast}(p)=(p,p^{2})\) (the constraint \(x_{1}+x_{2}\leqslant 2+p\)
is inactive at \(\bar x\)) with \(\Phi(p)=\{(0,0)\}\) constant in \(p\).
Hence \(S(p)=\{(p,p^{2})\}\) is single-valued.

\smallskip
\emph{Failure of \textup{(VDB)}.} For \(x\in H(p)\) near \(x^{\ast}(p)\),
\[
    \rho(p,x)=\dist(f(p,x),\Phi(p))=\sqrt 2\,\phi(p,x),
    \qquad
    \dist(x,S(p))=\|x-x^{\ast}(p)\|=\sqrt{\phi(p,x)},
\]
so the linear bound \(\dist(x,S(p))\leqslant\kappa\rho(p,x)\) demands
\(\sqrt{\phi}\leqslant\kappa\sqrt 2\,\phi\), which fails as
\(\phi\to 0\). The actual relation is the H\"older-\(\tfrac{1}{2}\)
bound \(\dist(x,S(p))=(2)^{-1/4}\rho(p,x)^{1/2}\), as predicted by
Remark~\ref{rem:VDB-sharp-fwd}, since \(\nabla_{x}f(\bar p,\bar x)=0\).

\smallskip
\emph{Direct computation of \(\DD S\).} With \(\dot p=1\) and the
first-order direction \(u=(1,0)\in DS(\bar p,\bar x)(1)\), the unique
realizer of \(x^{\ast}(t)=\bar x+tu+t^{2}w_{t}\) is \(w_{t}=(0,1)\) for
all \(t\), so \(\DD S(\bar p,\bar x,u)(1)=\{(0,1)\}\), a singleton.

\smallskip
\emph{Formula \eqref{eq:DDS-formula} overpredicts.} The constraint
\(x_{1}+x_{2}\leqslant 2+p\) is inactive at \(\bar x\), so
\(\DD H(\bar p,\bar x,u)(1)=\R^{2}\). The marginal map is constant,
\(\Phi(p)\equiv\{(0,0)\}\), giving \(\bar v=0\) and
\(\DD\Phi(\bar p,\bar y,\bar v)(1)=\{0\}\). Compute
\(\nabla_{x}f(\bar p,\bar x)=(\nabla_{x}\phi,\nabla_{x}\phi)=0\) and
\(d^{2}f(\bar p,\bar x)(1,u)=0\). The condition in
\eqref{eq:DDS-formula} becomes \(0+0\in\{0\}\), trivially true for
\emph{every} \(w\in\R^{2}\). Hence the formula returns
\[
    \mathcal W=\DD H(\bar p,\bar x,u)(1)=\R^{2}\;\;\supsetneq\;\;\{(0,1)\}=\DD S(\bar p,\bar x,u)(1).
\]
The strict inclusion shows that, without an inverse-stability
condition of \textup{(VDB)} type at \((\bar p,\bar x)\), the converse inclusion in the
proof of Theorem~\ref{thm:DDS-semideriv} may fail.
\end{example}

\section[Second-order sensitivity of S in the structured setting]{\texorpdfstring{{Second-order sensitivity of \(S\) in the structured setting}}{Second-order sensitivity of S in the structured setting}}%
\label{sec:paramsys}

This section specializes Theorem~\ref{thm:DDS-semideriv} to
the structured feasible map~\eqref{eq:Hstruct-def}, with closed
sets \(\Omega\subset X\) and \(D\subset Z\), using
Proposition~\ref{prop:DDH-str} and Theorem~\ref{thm:DDPhi-str}.

\begin{theorem}{Second-order semi-derivative of \(S\) in structured systems}{DDS-struct}
Let \((\bar p,\bar x)\in\gph S\), \(\bar y=f(\bar p,\bar x)\in\Phi(\bar p)\),
and \((p,u)\in P\times X\) with \(u\in DS(\bar p,\bar x)(p)\); define
\(\bar v\) by \eqref{eq:vbar}. Assume:

\begin{enumerate}[label=\textup{(\roman*)}]
\item \(\bar y\) is uniformly Henig efficient at \(\bar p\), and
      conditions \((A_1)\)--\((A_6)\) of
      Subsection~\ref{subsec:sec4-setup} hold with \(v=\bar v\);
\item \textup{(VDB)} holds at \((\bar p,\bar x)\).
\end{enumerate}
Then \(S\) is second-order semi-derivable at \((\bar p,\bar x)\) in
direction \((p,u)\), and
\begin{equation}\label{eq:DDS-struct}
    \DD S(\bar p,\bar x,u)(p)
    = \left\{w\in T^{2}(\Omega,\bar x,u)\,:\,
       \begin{array}{l}
          \nabla_{x}g(\bar p,\bar x)w+d^{2}g(\bar p,\bar x)(p,u)\\
          \quad\in T^{2}(D,g(\bar p,\bar x),\nabla g(\bar p,\bar x)(p,u)),\\[2pt]
          \nabla_{x}f(\bar p,\bar x)w+d^{2}f(\bar p,\bar x)(p,u)\\
          \quad\in \DD\Phi(\bar p,\bar y,\bar v)(p)
       \end{array}\right\},
\end{equation}
where \(\DD\Phi(\bar p,\bar y,\bar v)(p)\) is given explicitly by
\eqref{eq:DDPhi-str}.
\end{theorem}

\begin{proof}
Since \(\bar y=f(\bar p,\bar x)\) and
\(\bar x\in H(\bar p)\), we have
\(\bar x\in\Delta_0(\bar p,\bar y)\). Condition \((A_2)\) in
assumption \textup{(i)}, applied at \(\bar x\), gives that \(f\) and
\(g\) are \(C^1\) near \((\bar p,\bar x)\). Condition \((A_4)\) at
\(\bar x\) gives Robinson metric regularity of \(H\)
along \(\Omega\) at \((\bar p,\bar x)\).

Since \(u\in DS(\bar p,\bar x)(p)\),
Proposition~\ref{prop:compat1} gives
\[
 u\in DH(\bar p,\bar x)(p),
 \qquad
 \bar v=\nabla f(\bar p,\bar x)(p,u)
 \in D\Phi(\bar p,\bar y)(p).
\]
Thus \(u\in\Delta_1(\bar p,\bar x,p,\bar v)\). Condition \((A_3)\) in
assumption \textup{(i)}, applied at \((\bar x,u)\), gives second-order
semi-derivability of \(f\) and \(g\) in direction \((p,u)\),
second-order regularity of \(\Omega\) at \(\bar x\) in direction \(u\),
and second-order regularity of \(D\) at \(g(\bar p,\bar x)\) in direction
\(\nabla g(\bar p,\bar x)(p,u)\).

It remains to verify the first-order compatibility in
\eqref{eq:DDH-str-fo-compat}. From
\(u\in DH(\bar p,\bar x)(p)\), choose \(t_k\downarrow0\) and
\(u_k\to u\) such that
\(\bar x+t_ku_k\in H(\bar p+t_kp)\). The structured representation of
\(H\) gives
\[
 \bar x+t_ku_k\in\Omega,
 \qquad
 g(\bar p+t_kp,\bar x+t_ku_k)\in D.
\]
The first membership yields \(u\in T(\Omega,\bar x)\). Since \(g\) is
\(C^1\) near \((\bar p,\bar x)\),
\[
 g(\bar p+t_kp,\bar x+t_ku_k)
 =g(\bar p,\bar x)
  +t_k\nabla g(\bar p,\bar x)(p,u)+o(t_k).
\]
The second membership therefore yields
\(\nabla g(\bar p,\bar x)(p,u)\in
T(D,g(\bar p,\bar x))\). Hence \eqref{eq:DDH-str-fo-compat} holds, and
Proposition~\ref{prop:DDH-str} gives \eqref{eq:DDH-str}, which coincides
with the constraint part of the bracketed set in \eqref{eq:DDS-struct}.

The second relation supplied by Proposition~\ref{prop:compat1} means
\((p,\bar v)\in\gph D\Phi(\bar p,\bar y)\). Together with
conditions \((A_1)\)--\((A_6)\) in
assumption \textup{(i)}, this enables
Theorem~\ref{thm:DDPhi-str}. It gives \eqref{eq:DDPhi-str} and the
second-order semi-derivability of \(\Phi\) at \((\bar p,\bar y)\) in
direction \((p,\bar v)\). We have now verified hypotheses
\textup{(i)--(iii)} of Theorem~\ref{thm:DDS-semideriv}; assumption
\textup{(ii)} above is its hypothesis \textup{(iv)}. Applying that theorem
and substituting \eqref{eq:DDH-str} into \eqref{eq:DDS-formula} gives
\eqref{eq:DDS-struct} and the second-order semi-derivability of \(S\).
\end{proof}

Robinson metric regularity in \((A_4)\) follows from standard
Mangasarian--Fromovitz/Robinson conditions
\cite[Sect.~2.3.3]{BonnansShapiro2000}, and convex polyhedra are
second-order regular. Remark~\ref{rem:VDB-sharp-fwd} gives a sharp
scalarization criterion for \textup{(VDB)}. Constant-rank conditions in
\cite{Bondarevsky2016,MinchenkoStakhovski2011,MinchenkoTarakanov2015}
provide further checks for \(H\).

\begin{remark}{Sufficient conditions for Robinson metric regularity}{RMR-suff}
For the structured map~\eqref{eq:Hstruct-def}, if \(\Omega\) and \(D\) are convex and \(g\in C^{1}\) near \((\bar p,\bar x)\), the Robinson constraint qualification
\[
   \nabla_{x}g(\bar p,\bar x)\,T(\Omega,\bar x)-T\bigl(D,g(\bar p,\bar x)\bigr)=Z
\]
implies Robinson metric regularity along \(\Omega\) at \((\bar p,\bar x)\) via \cite[Thm.~2.87]{BonnansShapiro2000}. For an inequality/equality system \(H(p)=\{x\in\Omega:h_{i}(p,x)\leqslant0\;(i\in I),\;h_{j}(p,x)=0\;(j\in I_{0})\}\), this Robinson CQ reduces to the MFCQ analogue: surjectivity of the equality Jacobian \(\nabla_{x}(h_{j})_{j\in I_{0}}\) on \(T(\Omega,\bar x)\), plus a common direction \(d\in T(\Omega,\bar x)\) with \(\nabla_{x}h_{j}(\bar p,\bar x)d=0\) (\(j\in I_{0}\)) and \(\nabla_{x}h_{i}(\bar p,\bar x)d<0\) at every active inequality.
\end{remark}

\section[Applications]{\texorpdfstring{{Applications}}{Applications}}\label{sec:applications}

\subsection{Parametric multi-objective portfolio optimization}%
\label{subsec:app-portfolio}

For the portfolio model,
Corollary~\ref{cor:DDS-portfolio-verif} reduces \eqref{eq:DDS-struct} to
active-set tests, while Example~\ref{ex:VDB-portfolio-recovery}
computes the \textup{(VDB)} constant \(\kappa=\sqrt6\) directly.

Let \(x\in\R^{n}\) be the portfolio weights and \(p\) the
parameter. Assume \(\mu,\Sigma,s,A,b\) are \(C^{2}\) near \(\bar p\), with
\(\Sigma(p)\) symmetric \textup{(}otherwise replace it by
\((\Sigma(p)+\Sigma(p)^{\top})/2\)\textup{)}, so
\(\nabla_x(x^\top\Sigma(p)x)=2\Sigma(p)x\). Fix a benchmark \(x^0\) and
bounds \(\ell<x^{\max}\) componentwise. Consider
the four-component vector objective
\begin{equation}\label{eq:portfolio-obj}
    f(p,x)
    :=
    \bigl(
        -\mu(p)^{\top}x,\;\,
        x^{\top}\Sigma(p)x,\;\,
        -s(p)^{\top}x,\;\,
        \|x-x^{0}\|_{2}^{2}
    \bigr),
\end{equation}
{subject to}
\begin{equation}\label{eq:portfolio-feas}
    \mathbf 1^{\top}x=1,\qquad
    \ell\leqslant x\leqslant x^{\max},\qquad
    A(p)x\leqslant b(p).
\end{equation}
The components minimize negative expected return, variance,
negative ESG aggregate \textup{(}thus maximizing ESG\textup{)}, and tracking
error to \(x^0\), with order cone \(C=\R^4_+\). Related portfolio models appear
in \cite{GarciaBernabeu2024,Kovalenko2025,MohsenyTonekabony2025,Varmaz2024}.

Take
\[
    \Omega:=\bigl\{x\in\R^{n}\,:\,\mathbf 1^{\top}x=1,\;\,\ell\leqslant x\leqslant x^{\max}\bigr\},
    \qquad
    g(p,x):=A(p)x-b(p),\qquad
    D:=\R^{q}_{-}.
\]
Since \(\Omega\) and \(D\) are closed convex polyhedra and \(g\) is \(C^2\),
the model fits \eqref{eq:Hstruct-def}, and both sets are second-order regular
at every point and direction. A sufficient condition for Robinson metric
regularity of \(H\) along \(\Omega\) at \((\bar p,\bar x)\) is
\begin{equation}\label{eq:portfolio-CQ}
   0\in\Int\bigl(g(\bar p,\bar x)+\nabla_{x}g(\bar p,\bar x)(\Omega-\bar x)-D\bigr).
\end{equation}
In this polyhedral model, LICQ for the active budget, box, and linear
inequality constraints implies \eqref{eq:portfolio-CQ}. The first-order
compatibility condition \eqref{eq:DDH-str-fo-compat} follows from
\(u\in DS(\bar p,\bar x)(\dot p)\subset DH(\bar p,\bar x)(\dot p)\).
The set \(S(p)\) consists of efficient portfolios, and
\(\Phi(p)=f(p,S(p))\) is the efficient return/risk/ESG/tracking-error
frontier.

For \(\bar x\in \Omega\), \(u\in T(\Omega,\bar x)\), and \(\dot p\in P\), set
\(\bar v_{g}:=\nabla g(\bar p,\bar x)(\dot p,u)\) and define
\begin{align*}
   I_{\ell}&:=\{i\,:\,\bar x_{i}=\ell_{i}\}, &
   I_{\max}&:=\{i\,:\,\bar x_{i}=x^{\max}_{i}\},\\
   I_{g}&:=\{j\,:\,[A(\bar p)\bar x-b(\bar p)]_{j}=0\},\\
   I_{\ell}^{0}(u)&:=\{i\in I_{\ell}\,:\,u_{i}=0\}, &
   I_{\max}^{0}(u)&:=\{i\in I_{\max}\,:\,u_{i}=0\},\\
   I_{g}^{0}(\dot p,u)&:=\{j\in I_{g}\,:\,[\bar v_{g}]_{j}=0\}.
\end{align*}

For a general feasible point \(x\in\Omega\) we write
\(I_{\ell}(x),I_{\max}(x),I_{g}(x)\) for the active sets formed at \(x\), so
that \(I_{\ell}=I_{\ell}(\bar x)\), \(I_{\max}=I_{\max}(\bar x)\), and
\(I_{g}=I_{g}(\bar x)\).

\begin{lemma}{Polyhedral cones and explicit Taylor data for the portfolio model}{port-derivs}
With the active-set notation above,
\begin{align*}
   T(\Omega,\bar x) &=\bigl\{d\,:\,\mathbf 1^{\top}d=0,\;d_{i}\geqslant 0\;\forall i\in I_{\ell},\;d_{i}\leqslant 0\;\forall i\in I_{\max}\bigr\},\\
   T^{2}(\Omega,\bar x,u) &=\bigl\{w\,:\,\mathbf 1^{\top}w=0,\;w_{i}\geqslant 0\;\forall i\in I_{\ell}^{0}(u),\;w_{i}\leqslant 0\;\forall i\in I_{\max}^{0}(u)\bigr\},\\
   T(D,g(\bar p,\bar x)) &=\bigl\{e\,:\,e_{j}\leqslant 0\;\forall j\in I_{g}\bigr\},\\
   T^{2}(D,g(\bar p,\bar x),\bar v_{g}) &=\bigl\{r\,:\,r_{j}\leqslant 0\;\forall j\in I_{g}^{0}(\dot p,u)\bigr\}.
\end{align*}
Under the convention of Remark~\ref{rem:2nd-order-conv}, the Taylor
data of \(g(p,x)=A(p)x-b(p)\) and \(f\) of \eqref{eq:portfolio-obj} at
\((\bar p,\bar x)\) in direction \((\dot p,u)\), and the
\(x\)-Jacobians applied to a direction \(w\in\R^{n}\), are given by
\eqref{eq:port-derivs-g}--\eqref{eq:port-derivs-d2f} below.
\begin{align}\label{eq:port-derivs-g}
   \nabla_{x}g(\bar p,\bar x)w &= A(\bar p)w,\notag\\
   \nabla g(\bar p,\bar x)(\dot p,u)
   &= A(\bar p)u + (\nabla_{p}A(\bar p)\dot p)\bar x - \nabla_{p}b(\bar p)\dot p,\\
   d^{2}g(\bar p,\bar x)(\dot p,u)
   &= (\nabla_{p}A(\bar p)\dot p)u + d^{2}A(\bar p)(\dot p)\bar x - d^{2}b(\bar p)(\dot p).\notag
\end{align}
\begin{equation}\label{eq:port-derivs-Jxf}
   \nabla_{x}f(\bar p,\bar x)w
   = \bigl(-\mu(\bar p)^{\top}w,\;2\bar x^{\top}\Sigma(\bar p)w,\;-s(\bar p)^{\top}w,\;2(\bar x-x^{0})^{\top}w\bigr).
\end{equation}
\begin{equation}\label{eq:port-derivs-df}
   \nabla f(\bar p,\bar x)(\dot p,u)
   = \begin{pmatrix}
       -\mu(\bar p)^{\top}u-(\nabla_{p}\mu(\bar p)\dot p)^{\top}\bar x\\
       2\bar x^{\top}\Sigma(\bar p)u+\bar x^{\top}(\nabla_{p}\Sigma(\bar p)\dot p)\bar x\\
       -s(\bar p)^{\top}u-(\nabla_{p}s(\bar p)\dot p)^{\top}\bar x\\
       2(\bar x-x^{0})^{\top}u
   \end{pmatrix}.
\end{equation}
\begin{equation}\label{eq:port-derivs-d2f}
   d^{2}f(\bar p,\bar x)(\dot p,u)
   = \begin{pmatrix}
       -(\nabla_{p}\mu(\bar p)\dot p)^{\top}u-(d^{2}\mu(\bar p)(\dot p))^{\top}\bar x\\
       u^{\top}\Sigma(\bar p)u+2\bar x^{\top}(\nabla_{p}\Sigma(\bar p)\dot p)u+\bar x^{\top}d^{2}\Sigma(\bar p)(\dot p)\bar x\\
       -(\nabla_{p}s(\bar p)\dot p)^{\top}u-(d^{2}s(\bar p)(\dot p))^{\top}\bar x\\
       \|u\|^{2}
   \end{pmatrix}.
\end{equation}
\end{lemma}

\begin{proof}
The tangent cones follow from the active polyhedral
inequalities, and
\(T^{2}(Q,q,h)=T(T(Q,q),h)\) for polyhedral \(Q\)
\cite[Prop.~3.34]{BonnansShapiro2000}.
Expanding \(A(\bar p+t\dot p)(\bar x+tu)-b(\bar p+t\dot p)\) gives
\eqref{eq:port-derivs-g}; componentwise expansion of
\eqref{eq:portfolio-obj} gives
\eqref{eq:port-derivs-Jxf}--\eqref{eq:port-derivs-d2f}.
\end{proof}

\begin{proposition}{First-order semi-derivability of polyhedral constraint systems}{H-semider-polyhedral}
Let \((\bar p,\bar x)\in\gph H\). Assume:
\begin{enumerate}[label=\textup{(\roman*)},nosep]
\item \(\Omega\) and \(D\) are convex polyhedra;
\item \(g\) is continuously differentiable near \((\bar p,\bar x)\);
\item \(H\) is Robinson metrically regular along \(\Omega\) at
      \((\bar p,\bar x)\).
\end{enumerate}
Then, for every \(p\in P\), \(H\) is semi-derivable at
\((\bar p,\bar x)\) in direction \(p\), and
\begin{equation}\label{eq:DH-polyhedral}
 DH(\bar p,\bar x)(p)=D_\ell H(\bar p,\bar x)(p)
 =\{d\in T(\Omega,\bar x):
       \nabla g(\bar p,\bar x)(p,d)\in T(D,g(\bar p,\bar x))\}.
\end{equation}
\end{proposition}

\begin{proof}
Denote the right-hand side of \eqref{eq:DH-polyhedral} by
\(L(p)\). A realizing sequence and the first-order expansion of \(g\)
give \(DH(\bar p,\bar x)(p)\subset L(p)\). Conversely, fix
\(d\in L(p)\) and any \(t_k\downarrow0\). For a closed convex
polyhedron \(Q\), \(h\in T(Q,q)\) implies \(q+th\in Q\) for all small
\(t\geqslant0\); hence \(\bar x+t_kd\in\Omega\) and
\(g(\bar p,\bar x)+t_k\nabla g(\bar p,\bar x)(p,d)\in D\). Therefore
\(\dist(g(\bar p+t_kp,\bar x+t_kd),D)=o(t_k)\). Robinson metric
regularity supplies \(x_k\in H(\bar p+t_kp)\) with
\(x_k=\bar x+t_kd+o(t_k)\), so \(d\in D_\ell H(\bar p,\bar x)(p)\).
Thus \(L(p)\subset D_\ell H\subset DH\subset L(p)\).
\end{proof}

For a polyhedral inequality/equality system \(H(p)=\{x\in\Omega: h_i(p,x)\leqslant0\;(i\in I),\;h_j(p,x)=0\;(j\in I_0)\}\), the sets \(\Omega\) and \(D\) are convex polyhedra. For a convex polyhedron \(Q\), \(q\in Q\), and \(h\in T(Q,q)\), one has \(T^{2}(Q,q,h)=T\bigl(T(Q,q),h\bigr)\) by \cite[Prop.~3.34]{BonnansShapiro2000}; since \(Q-q\) coincides with the cone \(T(Q,q)\) near the origin, the scaling by \(t\) drops out and the asymptotic set coincides with it,
\begin{equation}\label{eq:poly-Tpp}
 T''(Q,q,h)=T^{2}(Q,q,h)=T\bigl(T(Q,q),h\bigr),
\end{equation}
so \((A_6)\) may be tested with \(T^{2}\) in place of \(T''\). Hence every cone entering \eqref{eq:DDS-struct} is polyhedral, and \eqref{eq:DDS-struct} takes an explicit form.

\begin{corollary}{Polyhedral inequality/equality system}{ineqeq}
Assume, in addition to the hypotheses of Theorem~\ref{thm:DDS-struct},
that the constraint has the form
\[
    H(p) = \{x\in \Omega\,:\,h_{i}(p,x)\leqslant 0\;(i\in I),\; h_{j}(p,x)=0\;(j\in I_{0})\},
\]
with each \(h_{i},h_{j}\) being \(C^{1}\) and second-order semi-derivable
at \((\bar p,\bar x)\) in direction \((p,u)\). Denote by
\(I(\bar p,\bar x):=\{i\in I\,:\,h_{i}(\bar p,\bar x)=0\}\) the active index set.
Then \(\DD S(\bar p,\bar x,u)(p)\) is characterized as in
\eqref{eq:DDS-struct} with \(g=(h_{i})_{i\in I\cup I_{0}}\) and
\(D=\R^{|I|}_{-}\times\{0\}^{|I_{0}|}\). The second-order tangent set
\(T^{2}(D,g(\bar p,\bar x),\nabla g(\bar p,\bar x)(p,u))\) then takes the
explicit form
\[
    \bigl\{d\in\R^{|I|+|I_{0}|}\,:\,
      d_{i}\leqslant 0\text{ for all }i\in I(\bar p,\bar x)\text{ with }\nabla h_{i}(\bar p,\bar x)(p,u)=0,\;
      d_{j}=0\text{ for }j\in I_{0}\bigr\}.
\]
\end{corollary}

\begin{proof}
Substitute
\(D=\R_-^{|I|}\times\{0\}^{|I_0|}\) in
Theorem~\ref{thm:DDS-struct}. First-order compatibility gives zero on
the equality block and nonpositive components at active inequalities.
For the polyhedral second-order tangent set, an active inequality with
zero first-order component contributes \(d_i\leqslant0\), whereas a
strictly negative or inactive component is free; every equality
component contributes \(d_j=0\). This is the displayed formula.
\end{proof}

\begin{corollary}{Portfolio second-order Dini derivative under verifiable conditions}{DDS-portfolio-verif}
Let \((\bar p,\bar x)\in\gph S\), \(\bar y=f(\bar p,\bar x)\),
\((\dot p,u)\in P\times\R^{n}\) with \(u\in DS(\bar p,\bar x)(\dot p)\), and
\(\bar v=\nabla f(\bar p,\bar x)(\dot p,u)\). Suppose \(\bar y\) is uniformly
Henig efficient at \(\bar p\).
Assume:
\begin{enumerate}[label=\textup{(\roman*)}]
\item for every \(\tilde x\in\Delta_{0}(\bar p,\bar y)\), the active gradients
      \(\{\mathbf 1\}\cup\{e_{i}:i\in I_{\ell}(\tilde x)\cup I_{\max}(\tilde x)\}\cup\{A(\bar p)_{j,:}^{\top}:j\in I_{g}(\tilde x)\}\)
      are linearly independent in \(\R^{n}\) \textup{(}LICQ\textup{)}, and, with
      \(T(H(\bar p),\tilde x)=\{d\in T(\Omega,\tilde x):A(\bar p)d\in T(D,g(\bar p,\tilde x))\}\),
      \[
         \{d\in T(H(\bar p),\tilde x):\nabla_{x}f(\bar p,\tilde x)d\in -C\}=\{0\};
      \]
\item for every \(\tilde x\in\Delta_{0}(\bar p,\bar y)\) and
      \(\tilde u\in\Delta_{1}(\bar p,\tilde x,\dot p,\bar v)\), with
      \(\tilde v_{g}=\nabla g(\bar p,\tilde x)(\dot p,\tilde u)\),
      \[
         \{\xi\in T^{2}(\Omega,\tilde x,\tilde u):A(\bar p)\xi\in T^{2}(D,g(\bar p,\tilde x),\tilde v_{g}),\ \nabla_{x}f(\bar p,\tilde x)\xi\in -C\}=\{0\};
      \]
\item \textup{(VDB)} holds at \((\bar p,\bar x)\).
\end{enumerate}
Then \(S\) is second-order semi-derivable at \((\bar p,\bar x)\) in direction
\((\dot p,u)\) and, with the index sets and Taylor data of Lemma~\ref{lem:port-derivs},
\begin{equation}\label{eq:DDS-portfolio}
   \DD S(\bar p,\bar x,u)(\dot p)
   =\left\{w\in\R^{n}\,:\,
      \begin{array}{l}
        \mathbf 1^{\top}w=0,\\
        w_{i}\geqslant 0\;\forall i\in I_{\ell}^{0}(u),\quad
        w_{i}\leqslant 0\;\forall i\in I_{\max}^{0}(u),\\[2pt]
        \bigl[A(\bar p)w+(\nabla_{p}A(\bar p)\dot p)u\\
        \;\;+d^{2}A(\bar p)(\dot p)\bar x-d^{2}b(\bar p)(\dot p)\bigr]_{j}\leqslant 0\\
        \quad\forall j\in I_{g}^{0}(\dot p,u),\\[2pt]
        \nabla_{x}f(\bar p,\bar x)w+d^{2}f(\bar p,\bar x)(\dot p,u)\\
        \quad\in\DD\Phi(\bar p,\bar y,\bar v)(\dot p)
      \end{array}\right\},
\end{equation}
where the four-component vectors \(\nabla_{x}f(\bar p,\bar x)w\) and
\(d^{2}f(\bar p,\bar x)(\dot p,u)\) are given componentwise by
Lemma~\ref{lem:port-derivs}.
\end{corollary}

\begin{proof}
The data are \(C^2\), and the polyhedra \(\Omega,D\) are
second-order regular
\cite[Prop.~3.34]{BonnansShapiro2000};
compactness gives \((A_1)\). LICQ in \textup{(i)} implies
\eqref{eq:portfolio-CQ} \cite[Sect.~2.3.4]{BonnansShapiro2000}, so
Proposition~\ref{prop:H-semider-polyhedral} applies. The membership
\(u\in DS(\bar p,\bar x)(\dot p)\) gives first-order compatibility.
The tangent formula \cite[Cor.~2.91]{BonnansShapiro2000} and
\eqref{eq:port-derivs-Jxf} identify \textup{(i)} with
\((A_5)\); since \(\Omega\) and \(D\) are polyhedral,
\eqref{eq:poly-Tpp} lets \((A_6)\) be tested with \(T^{2}\), which
identifies \textup{(ii)} with \((A_6)\). Item \textup{(iii)} is \textup{(VDB)}; applying
Theorem~\ref{thm:DDS-struct} and Lemma~\ref{lem:port-derivs} yields
\eqref{eq:DDS-portfolio}.
\end{proof}

Example~\ref{ex:VDB-portfolio-recovery} gives a
data-level instance of condition \textup{(iii)}, while
Proposition~\ref{prop:canonical-K} supplies the witnessing cone under its
common-scalarizer condition.

\begin{example}{Three-asset portfolio recovery}{VDB-portfolio-recovery}
Let \(r_0>0\), let \(\beta:(-r_0,r_0)\to(0,1)\) be \(C^2\) with
\(\beta(0)=1/10\), and restrict the portfolio model to \(p\in\R\),
\(n=3\), and \(q=2\). Set
\[
 x^0=e_1,\qquad \ell=0,\qquad x^{\max}=\mathbf 1,\qquad
 A(p)=\begin{pmatrix}1&-1&0\\0&1&0\end{pmatrix},\qquad
 b(p)=\binom{\beta(p)}{9/10},
\]
\[
 a(p):=\frac{1+\beta(p)}2,\qquad
 c(p):=\frac{1-\beta(p)}2,\qquad
 x^*(p):=(a(p),c(p),0).
\]
Let \(\mu(p)=e_1\), \(s(p)=(2,1,0)^\top\), and, for some
\(\rho\geqslant1\),
\[
 \Sigma(p)=
 \begin{pmatrix}
  1/a(p)&0&\rho\\
  0&1/c(p)&\rho\\
  \rho&\rho&\rho^2+1
 \end{pmatrix}.
\]
Then \(\Sigma(p)\) is positive definite, \(\mu(p)\) and \(s(p)\) are
not collinear, and
\[
\begin{gathered}
 S(p)=\{x^*(p)\},\qquad \Phi(p)=\{f(p,x^*(p))\},\\
 \dist(x,S(p))\leqslant\sqrt6\,\dist(f(p,x),\Phi(p))
 \qquad\forall x\in H(p).
\end{gathered}
\]
Thus \textup{(VDB)} holds throughout the family with constant \(\sqrt6\).
The common scalarizer \(\lambda=\mathbf1\) and
\(K=K_{\mathbf1,1/2}\) of Proposition~\ref{prop:canonical-K} certify
uniform Henig efficiency.

To see this, fix \(p\), abbreviate \(a=a(p)\), \(c=c(p)>0\), and set
\(k^1=(-1,1,0)\), \(k^2=(-1,-1,2)\). The Schur complement of the
leading \(2\times2\) block of \(\Sigma(p)\) equals one, so
\(\Sigma(p)\) is positive definite; also
\(\Sigma(p)x^*(p)=(1,1,\rho)^\top\).
For \(x\in H(p)\), put \(\delta=a-x_1\) and \(z=x_3\). The budget and
spread constraints give \(0\leqslant z\leqslant2\delta\), and hence
\[
 x-x^*(p)=\alpha_1k^1+\alpha_2k^2,
 \qquad \alpha_1:=\delta-\frac z2\geqslant0,
 \quad \alpha_2:=\frac z2\geqslant0.
\]
The objective derivatives on \((k^1,k^2)\) are
\((1,0,1,4c)\) and \((1,4(\rho-1),3,0)\), respectively. Convexity of
the objective components therefore gives
\(f(p,x)-f(p,x^*(p))\in\R^4_+\), and its first component is
\(\alpha_1+\alpha_2=\delta\). It is positive unless \(x=x^*(p)\), so
the two singleton identities follow. Moreover,
\[
 \|x-x^*(p)\|
 \leqslant\sqrt6(\alpha_1+\alpha_2)
 =\sqrt6\,\delta
 \leqslant\sqrt6\,\|f(p,x)-f(p,x^*(p))\|,
\]
which is the stated \textup{(VDB)} estimate. Componentwise domination makes
\(\lambda=\mathbf1\) a common scalarizer; since
\(m_{\R^4_+}(\mathbf1)=1>1/2\), Proposition~\ref{prop:canonical-K}
gives the Henig cone. Finally, \(e_1\) and \((2,1,0)^\top\) are not
collinear.
\end{example}

\subsection{A two-branch example}\label{subsec:app-twobranch}

\begin{example}{A two-branch structured system satisfying \textup{(VDB)}}{twobranch-VDB}
Let \(P=\R\), \(X=Y=\R^2\), \(C=\R^2_+\), and
\[
 \Omega=[-1,2]\times\{0,1\},\qquad D=\R_-,\qquad
 \beta(p)=p+p^2,
\]
\[
 g(p,x)=\beta(p)-x_1,\qquad
 H(p)=\{x\in\Omega:g(p,x)\in D\},\qquad
 f(p,x)=(x_1,x_1).
\]
For \(p\) near \(0\), direct minimization gives
\[
 \mathcal F(p)=\{(z,z):\beta(p)\leqslant z\leqslant2\},\qquad
 \Phi(p)=\{(\beta(p),\beta(p))\},
\]
and
\begin{equation}\label{eq:twobranch-S}
 S(p)=\{(\beta(p),0),(\beta(p),1)\}.
\end{equation}
Thus, at \(\bar p=0\) and \(\bar y=(0,0)\),
\[
 \Delta_0(0,\bar y)=\{(0,0),(0,1)\},\qquad
 |\Delta_0(0,\bar y)|=2.
\]

\emph{Exact (VDB) certificate.}
Fix \(i\in\{0,1\}\), set \(\bar x_i=(0,i)\), and choose a neighborhood
\(W_i\) that meets \(\Omega\) only on the branch \(x_2=i\). After shrinking
a neighborhood \(V\) of \(0\), \((\beta(p),i)\in W_i\) for every \(p\in V\).
For \(x=(x_1,i)\in H(p)\cap W_i\), put
\(\delta=x_1-\beta(p)\geqslant0\). Then
\begin{equation}\label{eq:twobranch-distances}
 \dist\bigl(x,S(p)\cap W_i\bigr)=\delta,
 \qquad
 \dist\bigl(f(p,x),\Phi(p)\bigr)=\sqrt2\,\delta.
\end{equation}
Hence \textup{(VDB)} holds at each \((0,\bar x_i)\) with the exact constant
\(\kappa=1/\sqrt2\). This also follows from
Remark~\ref{rem:VDB-sharp-fwd}: with
\(\lambda=(1,1)/\sqrt2\), \(m(p)=\sqrt2\,\beta(p)\), and
\(\alpha=\sqrt2\), conditions \eqref{eq:VDB-common-level}--\eqref{eq:VDB-sharp}
hold with scalar gap \(\sqrt2\,\delta\), and
\(\|\lambda\|/\alpha=1/\sqrt2\).

\emph{Uniform Henig efficiency and structured hypotheses.}
For \(\lambda=(1,1)/\sqrt2\),
\(m_C(\lambda)=1/\sqrt2\); choose \(\varepsilon=1/2\).
Every \((z,z)\in\mathcal F(p)\) satisfies
\[
 \langle\lambda,(z,z)-(\beta(p),\beta(p))\rangle
 =\sqrt2\,(z-\beta(p))\geqslant0.
\]
Proposition~\ref{prop:canonical-K} therefore supplies the common uniform
Henig witness \(K_{\lambda,1/2}\). The set \(\Omega\) is locally an affine
line at each \(\bar x_i\), while \(D\) is polyhedral, so both sets are
second-order regular in the directions used below. Moreover, \(f,g\) are
\(C^2\), compactness of \(\Omega\) implies \((A_1)\), and for
\(x\in\Omega\) near either branch,
\[
 \dist(x,H(p))=\max\{\beta(p)-x_1,0\}
 =\dist(g(p,x),D).
\]
Thus Robinson metric regularity holds with modulus one. The homogeneous
systems \((A_5)\) and \((A_6)\) reduce to a vector \((d_1,0)\) with
\(d_1\geqslant0\) and \((d_1,d_1)\in-C\), which implies
\(d_1=0\). Directly,
\(DH(0,\bar x_i)(1)=D_\ell H(0,\bar x_i)(1)
=\{(d_1,0):d_1\geqslant1\}\). Hence the remaining first-order
semi-derivability requirement also holds, and the structured hypotheses are
satisfied at both base branches.

\emph{Branchwise second-order calculation.}
Take the parameter direction \(p=1\). From \eqref{eq:twobranch-S}, for
each \(i\in\{0,1\}\),
\[
 DS(0,\bar x_i)(1)=\{(1,0)\},\qquad
 \DD S(0,\bar x_i,(1,0))(1)=\{(1,0)\},
\]
because \(\beta(t)=t+t^2\) and the paper uses the \(t^2\)-coefficient
convention of Remark~\ref{rem:2nd-order-conv}. The structured formula gives
the same result. At either \(\tilde x\in\Delta_0(0,\bar y)\),
\(\nabla g(0,\tilde x)(1,(1,0))=0\) and
\(d^2g(0,\tilde x)(1,(1,0))=1\). Consequently,
\[
 \Delta_1(0,\tilde x,1,(1,1))=\{(1,0)\},\qquad
 \DD H(0,\tilde x,(1,0))(1)
 =\{(w_1,0):w_1\geqslant1\}.
\]
Since \(d^2f=0\) and \(\nabla_xf(w_1,0)=(w_1,w_1)\), the two branches in
\eqref{eq:DDPhi-str} contribute the same ray. Its \(C\)-minimum is
\[
 \DD\Phi(0,\bar y,(1,1))(1)=\{(1,1)\}.
\]
Substitution into \eqref{eq:DDS-struct} leaves only \(w_1=1\), and hence
recovers \(\DD S(0,\bar x_i,(1,0))(1)=\{(1,0)\}\) on both branches.
This makes the set-valued recovery mechanism concrete: the base
efficient-value fiber contains two decision branches.
\end{example}

\section[Conclusion]{\texorpdfstring{{Conclusion}}{Conclusion}}\label{sec:conclusion}

This paper isolates the value-to-decision bound \textup{(VDB)}
as an inverse-stability mechanism: it transfers second-order sensitivity
of the efficient-value map \(\Phi\) to the efficient-decision map \(S\).
Proposition~\ref{prop:DDS-outer} gives the outer estimate, and Theorem~\ref{thm:DDS-semideriv} the exact Dini formula.
For structured
constraints \(H(p)=\{x\in\Omega:g(p,x)\in D\}\), the formulas for \(\DD H\)
and \(\DD\Phi\) yield a primitive-data formula for \(\DD S\), including a
polyhedral specialization. The portfolio family and two-branch example
verify this framework analytically. Future developments may investigate
verifiable \textup{(VDB)} conditions under weaker regularity
assumptions and their interaction with second-order sufficient conditions.

\section*{Acknowledgments}
AI tools based on Claude (Anthropic) and GPT (OpenAI) were
used to review manuscript drafts, assist in checking proofs and
references, and improve the prose. The authors independently verified
the mathematical arguments and references. The authors assume
responsibility for all content.
%

\bibliographystyle{siamplain}

\end{document}